\documentclass[11pt,oneside,english,reqno]{amsart}
\usepackage[T1]{fontenc}
\usepackage[latin9]{inputenc}
\usepackage{geometry}
\usepackage{color}
\usepackage{babel}
\usepackage{array}
\usepackage{refstyle}
\usepackage{float}
\usepackage{mathrsfs}
\usepackage{amsthm}
\usepackage{amstext}
\usepackage{amssymb}
\usepackage{graphicx}
\usepackage{esint}
\usepackage[all]{xy}
\PassOptionsToPackage{normalem}{ulem}
\usepackage{ulem}
\usepackage[unicode=true,pdfusetitle,
 bookmarks=true,bookmarksnumbered=false,bookmarksopen=false,
 breaklinks=false,pdfborder={0 0 1},backref=false,colorlinks=false]
 {hyperref}
\hypersetup{
 colorlinks=true,citecolor=blue,linkcolor=blue,linktocpage=true}

\makeatletter

\AtBeginDocument{\providecommand\secref[1]{\ref{sec:#1}}}
\providecommand{\tabularnewline}{\\}
\RS@ifundefined{subref}
  {\def\RSsubtxt{section~}\newref{sub}{name = \RSsubtxt}}
  {}
\RS@ifundefined{thmref}
  {\def\RSthmtxt{theorem~}\newref{thm}{name = \RSthmtxt}}
  {}
\RS@ifundefined{lemref}
  {\def\RSlemtxt{lemma~}\newref{lem}{name = \RSlemtxt}}
  {}

\numberwithin{equation}{section}
\numberwithin{figure}{section}
\numberwithin{table}{section}
\usepackage{enumitem}		
\theoremstyle{plain}
\newtheorem{thm}{\protect\theoremname}[section]
  \theoremstyle{definition}
  \newtheorem{defn}[thm]{\protect\definitionname}
  \theoremstyle{remark}
  \newtheorem{rem}[thm]{\protect\remarkname}
  \theoremstyle{plain}
  \newtheorem{lem}[thm]{\protect\lemmaname}
  \theoremstyle{plain}
  \newtheorem{cor}[thm]{\protect\corollaryname}
  \theoremstyle{plain}
  \newtheorem{conjecture}[thm]{\protect\conjecturename}
  \theoremstyle{remark}
  \newtheorem*{acknowledgement*}{\protect\acknowledgementname}

\allowdisplaybreaks

\providecommand{\MR}[1]{}

\usepackage{mathtools}

\renewcommand{\section}{%
\@startsection{section}{1}%
  \z@{.7\linespacing\@plus\linespacing}{.5\linespacing}%
  {\normalfont\scshape\centering\bfseries}}

\renewcommand{\subsection}{%
\@startsection{subsection}{2}%
  \z@{.5\linespacing\@plus.7\linespacing}{.5\linespacing}%
  {\normalfont\bfseries}}

\renewcommand{\subsubsection}{%
\@startsection{subsubsection}{2}%
  \z@{.5\linespacing\@plus.7\linespacing}{.5\linespacing}%
  {\normalfont\bfseries}}

\newref{sec}{%
      name      = \RSsectxt,
      names     = \RSsecstxt,
      Name      = \RSSectxt,
      Names     = \RSSecstxt,
      refcmd    = {\ifRSstar\S\fi\ref{#1}},
      rngtxt    = \RSrngtxt,
      lsttwotxt = \RSlsttwotxt,
      lsttxt    = \RSlsttxt}

\theoremstyle{definition}

\newref{thm}{
         name      = {theorem~},
         names     = {theorems~},
         Name      = {Theorem~},
         Names     = {Theorems~},
         rngtxt    = \RSrngtxt,
         lsttwotxt = \RSlsttxt,
         lsttxt    = \RSlsttxt
         }  
\newref{prop}{
         name      = {proposition~},
         names     = {propositions~},
         Name      = {Proposition~},
         Names     = {Propositions~},
         rngtxt    = \RSrngtxt,
         lsttwotxt = \RSlsttxt,
         lsttxt    = \RSlsttxt
         }   
         
\newref{cor}{
         name      = {corollary~},
         names     = {corollaries~},
         Name      = {Corollary~},
         Names     = {Corollaries~},
         rngtxt    = \RSrngtxt,
         lsttwotxt = \RSlsttxt,
         lsttxt    = \RSlsttxt
         }   
         \newref{rem}{
         name      = {remark~},
         names     = {remarks~},
         Name      = {Remark~},
         Names     = {Remarks~},
         rngtxt    = \RSrngtxt,
         lsttwotxt = \RSlsttxt,
         lsttxt    = \RSlsttxt
         }   
         \newref{ex}{
         name      = {example~},
         names     = {examples~},
         Name      = {Example~},
         Names     = {Examples~},
         rngtxt    = \RSrngtxt,
         lsttwotxt = \RSlsttxt,
         lsttxt    = \RSlsttxt
         }   
         
         \newref{def}{
         name      = {definition~},
         names     = {definitions~},
         Name      = {Definition~},
         Names     = {Definitions~},
         rngtxt    = \RSrngtxt,
         lsttwotxt = \RSlsttxt,
         lsttxt    = \RSlsttxt
         }   

\AtBeginDocument{
  
}

\makeatother

  \providecommand{\acknowledgementname}{Acknowledgement}
  \providecommand{\conjecturename}{Conjecture}
  \providecommand{\corollaryname}{Corollary}
  \providecommand{\definitionname}{Definition}
  \providecommand{\lemmaname}{Lemma}
  \providecommand{\remarkname}{Remark}
\providecommand{\theoremname}{Theorem}

\begin{document}

\title{Frames and Factorization of Graph Laplacians}

\author{Palle Jorgensen and Feng Tian}

\address{(Palle E.T. Jorgensen) Department of Mathematics, The University
of Iowa, Iowa City, IA 52242-1419, U.S.A. }

\email{palle-jorgensen@uiowa.edu}

\urladdr{http://www.math.uiowa.edu/\textasciitilde{}jorgen/}

\address{(Feng Tian) Department of Mathematics, Wright State University, Dayton,
OH 45435, U.S.A.}

\email{feng.tian@wright.edu}

\urladdr{http://www.wright.edu/\textasciitilde{}feng.tian/}



\maketitle
\pagestyle{myheadings}
\markright{Frames and Factorization of Graph Laplacians}
\begin{abstract}
Using functions from electrical networks (graphs with resistors assigned
to edges), we prove existence (with explicit formulas) of a canonical
Parseval frame in the energy Hilbert space $\mathscr{H}_{E}$ of a
prescribed infinite (or finite) network. Outside degenerate cases,
our Parseval frame is not an orthonormal basis. We apply our frame
to prove a number of explicit results: With our Parseval frame and
related closable operators in $\mathscr{H}_{E}$ we characterize the
Friedrichs extension of the $\mathscr{H}_{E}$-graph Laplacian. 

We consider infinite connected network-graphs $G=\left(V,E\right)$,
$V$ for vertices, and \emph{E} for edges. To every conductance function
$c$ on the edges $E$ of $G$, there is an associated pair $\left(\mathscr{H}_{E},\Delta\right)$
where $\mathscr{H}_{E}$ in an energy Hilbert space, and $\Delta\left(=\Delta_{c}\right)$
is the $c$-Graph Laplacian; both depending on the choice of conductance
function $c$. When a conductance function is given, there is a current-induced
orientation on the set of edges and an associated natural Parseval
frame in $\mathscr{H}_{E}$ consisting of dipoles. Now $\Delta$ is
a well-defined semibounded Hermitian operator in both of the Hilbert
$l^{2}\left(V\right)$ and $\mathscr{H}_{E}$. It is known to automatically
be essentially selfadjoint as an $l^{2}\left(V\right)$-operator,
but generally not as an $\mathscr{H}_{E}$ operator. Hence as an $\mathscr{H}_{E}$
operator it has a Friedrichs extension. In this paper we offer two
results for the Friedrichs extension: a characterization and a factorization.
The latter is via $l^{2}\left(V\right)$.
\end{abstract}
\tableofcontents{}

\section{Introduction}

We study infinite networks with the use of frames in Hilbert space.
While our results apply to finite systems, we will concentrate on
the infinite case because of its statistical significance.

By a network we mean a graph $G$ with vertices $V$ and edges $E$.
We assume that each vertex is connected by single edges to a finite
number of neighboring vertices, and that resistors are assigned to
the edges. From this we define an associated graph-Laplacian $\Delta$,
and a resistance metric on $V$.

The functions on $V$ of interest represent voltage distributions.
While there are a number of candidates for associated Hilbert spaces
of functions on $V$, the one we choose here has its norm-square equal
to the energy of the voltage function. This Hilbert space is denoted
$\mathscr{H}_{E}$, and it depends on an assigned conductance function
(= reciprocal of resistance.) We will further study an associated
graph Laplacian as a Hermitian semibounded operator with dense domain
in $\mathscr{H}_{E}$.

In our first result we identify a canonical Parseval frame in $\mathscr{H}_{E}$,
and we show that it is not an orthonormal basis except in simple degenerate
cases. The frame vectors (for $\mathscr{H}_{E}$) are indexed by oriented
edges $e$, a dipole vector for each $e$, and a current through $e$.

We apply our frame to complete a number of explicit results. We study
the Friedrichs extension of the graph Laplacian $\Delta$. And we
use our Parseval frame and related closable operators in $\mathscr{H}_{E}$
to give a factorization of the Friedrichs extension of $\Delta$.

Continuing earlier work \cite{Jor08,JoPe13,JoPe11a,JoPe11b,JoPe10,MS14,Fol14}
on analysis and spectral theory of infinite connected network-graphs
$G=\left(V,E\right)$, $V$ for vertices, and $E$ for edges, we study
here a new factorization for the associated graph Laplacians. Our
starting point is a fixed conductance function $c$ for $G$. It is
known that, to every prescribed conductance function $c$ on the edges
$E$ of $G$, there is an associated pair $\left(\mathscr{H}_{E},\Delta\right)$
where $\mathscr{H}_{E}$ in an energy Hilbert space, and $\Delta\left(=\Delta_{c}\right)$
is the $c$-Graph Laplacian; both depending on the choice of conductance
function $c$. For related papers on frames and discrete harmonic
analysis, see also \cite{AJSV13,JS12,CJ12b,CJ12a,JP12,Dur06,ST81,Ter78}
and the papers cited there.

It is also known that $\Delta$ is a well-defined semibounded Hermitian
operator in both of the Hilbert $l^{2}\left(V\right)$ and $\mathscr{H}_{E}$;
densely defined in both cases; and in each case with a natural domain,
see \cite{JoPe10}. As an $l^{2}\left(V\right)$-operator $\Delta$
has an $\infty\times\infty$ representation expressed directly in
terms of $\left(c,G\right)$, and it is further known that $\Delta$
is automatically be essentially selfadjoint as an $l^{2}\left(V\right)$-operator,
but generally not as an $\mathscr{H}_{E}$ operator, \cite{Jor08,JoPe13}.
Hence as an $\mathscr{H}_{E}$ operator it has a Friedrichs extension.
In this paper we offer two results for the $\mathscr{H}_{E}$ Friedrichs
extension: a characterization and a factorization. The latter is via
the Hilbert space $l^{2}\left(V\right)$.

We begin with the basic notions needed, and we then turn to our theorem
about Parseval frames: In \secref{eframe}, we show that, when a conductance
function is given, there is a current-induced orientation on the set
of edges and an associated natural Parseval frame in the energy Hilbert
space $\mathscr{H}_{E}$ with the frame vectors consisting of dipoles.

\section{\label{sec:setting}Basic Setting}

Overview of the details below: The graph Laplacian $\Delta$ (Definition
\ref{def:lap}) has an easy representation as a densely defined semibounded
operator in $l^{2}\left(V\right)$ via its matrix representation,
see Remark \ref{rem:lapl2}. To do this we use implicitly the standard
orthonormal (ONB) basis $\left\{ \delta_{x}\right\} $ in $l^{2}(V)$.
But in network problems, and in metric geometry, $l^{2}(V)$ is not
useful; rather we need the energy Hilbert space $\mathscr{H}_{E}$,
see Lemma \ref{lem:eframe}. 

The problem with this is that there is not an independent characterization
of the domain $dom\left(\Delta,\mathscr{H}_{E}\right)$ when $\Delta$
is viewed as an operator in $\mathscr{H}_{E}$ (as opposed to in $l^{2}(V)$);
other than what we do in Definition \ref{def:D}, i.e., we take for
its domain $\mathscr{D}_{E}$ = finite span of dipoles. This creates
an ambiguity with functions on $V$ versus vectors in $\mathscr{H}_{E}$.
Note, vectors in $\mathscr{H}_{E}$ are equivalence classes of functions
on $V$. In fact we will see that it is not feasible to aim to prove
properties about $\Delta$ in $\mathscr{H}_{E}$ without first introducing
dipoles; see Lemma \ref{lem:dipole} below. Also the delta-functions
$\left\{ \delta_{x}\right\} $ from the $l^{2}(V)$-ONB will typically
not be total in $\mathscr{H}_{E}$. In fact, the $\mathscr{H}_{E}$
ortho-complement of $\left\{ \delta_{x}\right\} $ in $\mathscr{H}_{E}$
consists of the harmonic functions in $\mathscr{H}_{E}$.

Let $V$ be a countable discrete set, and let $E\subset V\times V$
be a subset such that:
\begin{enumerate}
\item $\left(x,y\right)\in E\Longleftrightarrow\left(y,x\right)\in E$;
$x,y\in V$;
\item $\#\left\{ y\in V\:|\:\left(x,y\right)\in E\right\} $ is finite,
and $>0$ for all $x\in V$; 
\item $\left(x,x\right)\notin E$; and
\item $\exists\, o\in V$ s.t. for all $y\in V$ $\exists\, x_{0},x_{1},\ldots,x_{n}\in V$
with $x_{0}=o$, $x_{n}=y$, $\left(x_{i-1},x_{i}\right)\in E$, $\forall i=1,\ldots,n$.
(This property is called connectedness.)
\end{enumerate}
If a conductance function $c$ is given we require $c_{x_{i-1}x_{i}}>0$.
\begin{defn}
\label{def:cond}A function $c:E\rightarrow\mathbb{R}_{+}\cup\left\{ 0\right\} $
is called \uline{conductance function} if $c\left(e\right)\geq0$,
$\forall e\in E$, and if for all $x\in V$, and $\left(x,y\right)\in E$,
$c_{xy}>0$, and $c_{xy}=c_{yx}$. 
\end{defn}
Notation: If $x\in V$, we set 
\begin{equation}
c\left(x\right):=\sum_{y}c_{xy},\:\mbox{sum over \ensuremath{\left\{ y\in V\:\big|\:\left(x,y\right)\in E\right\} :=E\left(x\right)}.}\label{eq:cond}
\end{equation}
The summation in (\ref{eq:cond}) is denoted $x\sim y$. We say that
$x\sim y$ if $\left(x,y\right)\in E$. 
\begin{defn}
\label{def:lap}When $c$ is a conductance function (see Def. \ref{def:cond})
we set $\Delta=\Delta_{c}$ (the corresponding graph Laplacian)
\begin{align}
\left(\Delta u\right)\left(x\right) & =\sum_{y\sim x}c_{xy}\left(u\left(x\right)-u\left(y\right)\right)\label{eq:lap}\\
 & =c\left(x\right)u\left(x\right)-\sum_{y\sim x}c_{xy}u\left(y\right).\nonumber 
\end{align}
\end{defn}
\begin{rem}
\label{rem:lapl2}Given $\left(V,E,c\right)$ as above, and let $\Delta=\Delta_{c}$
be the corresponding graph Laplacian. With a suitable ordering on
$V$, we obtain the following banded $\infty\times\infty$ matrix-representation
for $\Delta$: 
\begin{equation}
\begin{bmatrix}c\left(x_{1}\right) & -c_{x_{1}x_{2}} & 0 & \cdots & \cdots & \cdots & \cdots & 0 & \cdots\\
-c_{x_{2}x_{1}} & c\left(x_{2}\right) & -c_{x_{2}x_{3}} & 0 & \cdots & \cdots & \cdots & \vdots & \cdots\\
0 & -c_{x_{3}x_{2}} & c\left(x_{3}\right) & -c_{x_{3}x_{4}} & 0 & \cdots & \cdots & \huge\mbox{0} & \cdots\\
\vdots & 0 & \ddots & \ddots & \ddots & \ddots & \vdots & \vdots & \cdots\\
\vdots & \vdots & \ddots & \ddots & \ddots & \ddots & 0 & \vdots & \cdots\\
\vdots & \huge\mbox{0} & \cdots & 0 & -c_{x_{n}x_{n-1}} & c\left(x_{n}\right) & -c_{x_{n}x_{n+1}} & 0 & \cdots\\
\vdots & \vdots & \cdots & \cdots & 0 & \ddots & \ddots & \ddots & \ddots
\end{bmatrix}\label{eq:lapm}
\end{equation}
(We refer to \cite{GLS12} for a number of applications of infinite
banded matrices.)

If $\#E\left(x\right)=2$ for all $x\in V$ where $E\left(x\right):=\left\{ y\in V\:|\:\left(x,y\right)\in E\right\} $
we say that the corresponding $\left(V,E,c\right)$ is nearest neighbor,
and in this case, the matrix representation takes the following form
relative the an ordering in $V$:
\begin{equation}
\xymatrix{o\ar@/_{1pc}/@{<->}[r]_{c_{ox_{1}}} & x_{1}\ar@/_{1pc}/@{<->}[r]_{c_{x_{1}x_{2}}} & x_{2} & \cdots & x_{n-1}\ar@/_{1pc}/@{<->}[r]_{c_{x_{n-1}x_{n}}} & x_{n}\ar@/_{1pc}/@{<->}[r]_{c_{x_{n}x_{n+1}}} & x_{n+1}\cdots}
\label{eq:nbh}
\end{equation}

\end{rem}
\begin{equation}
\begin{bmatrix}c\left(o\right) & -c_{ox_{1}} & 0 & \cdots & \cdots & \cdots & 0\\
-c_{ox_{1}} & c\left(x_{1}\right) & -c_{x_{1}x_{2}} & 0 & \cdots & \cdots & \vdots\\
0 & -c_{x_{1}x_{2}} & c\left(x_{2}\right) & c_{x_{2}x_{3}} & 0 &  & \huge\mbox{0}\\
\vdots & \ddots & \ddots & \ddots & \ddots & \ddots & \vdots\\
\vdots & \cdots & \ddots & -c_{x_{n}x_{n-1}} & c\left(x_{n}\right) & -c_{x_{n}x_{n+1}} & 0\\
\vdots & \cdots &  & \ddots & \ddots & \ddots & \ddots\\
\huge\mbox{0} & \cdots & \cdots & \cdots & 0
\end{bmatrix}\label{eq:nhd2}
\end{equation}

\begin{rem}[\textbf{Random Walk}]
\label{rem:rw} If $\left(V,E,c\right)$ is given as in Definition
\ref{def:lap}, then for $\left(x,y\right)\in E$, set 
\begin{equation}
p_{xy}:=\frac{c_{xy}}{c\left(x\right)}\label{eq:pxy}
\end{equation}
and note then $\left\{ p_{xy}\right\} $ in (\ref{eq:pxy}) is a system
of transition probabilities, i.e., $\sum_{y}p_{xy}=1$, $\forall x\in V$,
see Fig. \ref{fig:tp} below.

\begin{figure}[H]
\includegraphics[scale=0.5]{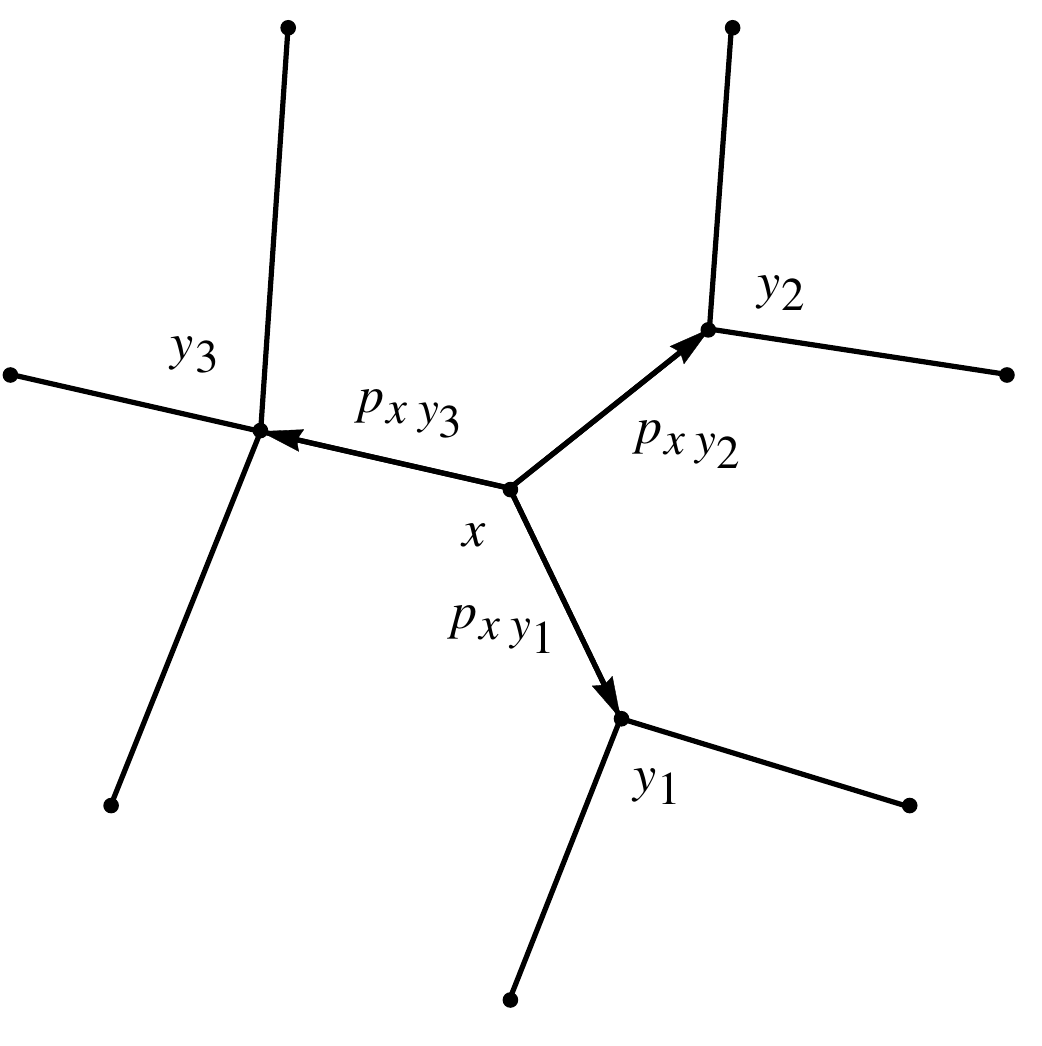}

\protect\caption{\label{fig:tp}Transition probabilities $p_{xy}$ at a vertex $x$
$\left(\mbox{in }V\right)$. }

\end{figure}

\end{rem}
A Markov-random walk on $V$ with transition probabilities $\left(p_{xy}\right)$
is said to be \emph{\uline{reversible}} iff $\exists$ a positive
function $\widetilde{c}$ on $V$ such that
\begin{equation}
\widetilde{c}\left(x\right)p_{xy}=\widetilde{c}\left(y\right)p_{yx},\;\forall\left(xy\right)\in E.\label{eq:mar1}
\end{equation}
 
\begin{lem}
There is a bijective correspondence between reversible Markov-walks
on the one hand, and conductance functions on the other.\end{lem}
\begin{proof}
If $c$ is a conductance function on $E$, see Definition \ref{def:cond},
then $\left(p_{xy}\right)$, defined in (\ref{eq:pxy}), is a reversible
walk. This follows from $c_{xy}=c_{yx}$. 

Conversely if (\ref{eq:mar1}) holds for a system of transition probabilities
$\left(p_{xy}=\mbox{Prob}\left(x\mapsto y\right)\right)$, then $c_{xy}:=\widetilde{c}\left(x\right)p_{xy}$
is a conductance function, where 
\[
\widetilde{c}\left(x\right)=\sum_{y\sim x}c_{xy}.
\]

\end{proof}
For results on reversible Markov chains, see e.g., \cite{LPP12}.

\section{\label{sec:eframe}Electrical Current as Frame Coefficients}

The role of the graph-network setting $\left(V,E,c,\mathscr{H}_{E}\right)$
introduced above is, in part, to model a family of electrical networks;
one application among others. Here $G$ is a graph with vertices $V$,
and edges $E$. Since we study large networks, it is helpful to take
$V$ infinite, but countable. Think of a network of resistors placed
on the edges in $G$. In this setting, the functions $v_{\left(x,y\right)}$
in $\mathscr{H}_{E}$, indexed by pairs of vertices, represent dipoles.
They measure voltage drop in the network through all possible paths
between the two vertices $x$ and $y$. Now the conduction function
$c$ is given, and so (electrical) current equals the product of conductance
and voltage drop; in this case voltage drop is computed over the paths
made up of edges from $x$ to $y$. For infinite systems $\left(V,E,c\right)$
the corresponding dipoles $v_{xy}$ are \uline{not} in $l^{2}\left(V\right)$,
but they are always in $\mathscr{H}_{E}$; see Lemma \ref{lem:dipole}
below. 

For a fixed function $u$ in $\mathscr{H}_{E}$ (voltage), and $e$
in $E$ we calculate the current $I\left(u,e\right)$, and we show
that these numbers yield \emph{\uline{frame coefficients}} in a
natural Parseval frame (for $\mathscr{H}_{E}$) where the frame vectors
making up the Parseval frame are $v_{e}$, $e=\left(x,y\right)$ in
$E$.

This result will be proved below. For general background references
on frames in Hilbert space, we refer to \cite{HJL13,KLZ09,CM13,SD13,KOPT13,EO13},
and for electrical networks, see \cite{Ana11,Gri10,Zem96,Bar93,Tet91,DS84,NW59}.
The facts on electrical networks we need are the laws of Kirchhoff
and Ohm, and our computation of the frame coefficients as electrical
currents is based on this, in part.
\begin{defn}
Let $\mathscr{H}$ be a Hilbert space with inner product denoted $\left\langle \cdot,\cdot\right\rangle $,
or $\left\langle \cdot,\cdot\right\rangle _{\mathscr{H}}$ when there
is more than one possibility to consider. Let $J$ be a countable
index set, and let $\left\{ w_{j}\right\} _{j\in J}$ be an indexed
family of non-zero vectors in $\mathscr{H}$. We say that $\left\{ w_{j}\right\} _{j\in J}$
is a \emph{\uline{frame}}\emph{ }for $\mathscr{H}$ iff (Def.)
there are two finite positive constants $b_{1}$ and $b_{2}$ such
that
\begin{equation}
b_{1}\left\Vert u\right\Vert _{\mathscr{H}}^{2}\leq\sum_{j\in J}\left|\left\langle w_{j},u\right\rangle _{\mathscr{H}}\right|^{2}\leq b_{2}\left\Vert u\right\Vert _{\mathscr{H}}^{2}\label{eq:en1}
\end{equation}
holds for all $u\in\mathscr{H}$. We say that it is a \emph{\uline{Parseval}}
frame if $b_{1}=b_{2}=1$. \end{defn}
\begin{lem}
\label{lem:eframe}If $\left\{ w_{j}\right\} _{j\in J}$ is a Parseval
frame in $\mathscr{H}$, then the (analysis) operator $A=A_{\mathscr{H}}:\mathscr{H}\longrightarrow l^{2}\left(J\right)$,
\begin{equation}
Au=\left(\left\langle w_{j},u\right\rangle _{\mathscr{H}}\right)_{j\in J}\label{eq:en2}
\end{equation}
is well-defined and isometric. Its adjoint $A^{*}:l^{2}\left(J\right)\longrightarrow\mathscr{H}$
is given by 
\begin{equation}
A^{*}\left(\left(\gamma_{j}\right)_{j\in J}\right):=\sum_{j\in J}\gamma_{j}w_{j}\label{eq:en3}
\end{equation}
and the following hold:
\begin{enumerate}
\item The sum on the RHS in (\ref{eq:en3}) is norm-convergent;
\item $A^{*}:l^{2}\left(J\right)\longrightarrow\mathscr{H}$ is co-isometric;
and for all $u\in\mathscr{H}$, we have 
\begin{equation}
u=A^{*}Au=\sum_{j\in J}\left\langle w_{j},u\right\rangle w_{j}\label{eq:en4}
\end{equation}
where the RHS in (\ref{eq:en4}) is norm-convergent. 
\end{enumerate}
\end{lem}
\begin{proof}
The details are standard in the theory of frames; see the cited papers
above. Note that (\ref{eq:en1}) for $b_{1}=b_{2}=1$ simply states
that $A$ in (\ref{eq:en2}) is isometric, and so $A^{*}A=I_{\mathscr{H}}=$
the identity operator in $\mathscr{H}$, and $AA^{*}=$ the projection
onto the range of $A$.
\end{proof}
When a conductance function $c:E\rightarrow\mathbb{R}_{+}\cup\left\{ 0\right\} $
is given, we consider the energy Hilbert space $\mathscr{H}_{E}$
(depending on $c$), with norm and inner product:
\begin{align}
\left\langle u,v\right\rangle _{\mathscr{H}_{E}} & :=\frac{1}{2}\underset{\left(x,y\right)\in E}{\sum\sum}c_{xy}\left(\overline{u\left(x\right)}-\overline{u\left(y\right)}\right)\left(v\left(x\right)-v\left(y\right)\right),\mbox{ and}\label{eq:en5}\\
\left\Vert u\right\Vert _{\mathscr{H}_{E}}^{2} & =\left\langle u,u\right\rangle _{\mathscr{H}_{E}}=\frac{1}{2}\underset{\left(x,y\right)\in E}{\sum\sum}c_{xy}\left|u\left(x\right)-u\left(y\right)\right|^{2}<\infty.\label{eq:en6}
\end{align}

We shall assume that $\left(V,E,c\right)$ is \emph{\uline{connected}}
(See Definition \ref{def:cond}); and it is shown that then (\ref{eq:en4})-(\ref{eq:en5})
define $\mathscr{H}_{E}$ as a Hilbert space of functions on $V$;
functions defined modulo constants; see also \secref{lemmas} below.

Further, for any pair of vertices $x,y\in V$, there is a unique dipole
vector $v_{xy}\in\mathscr{H}_{E}$ such that 
\begin{equation}
\left\langle v_{xy},u\right\rangle _{\mathscr{H}_{E}}=u\left(x\right)-u\left(y\right)\label{eq:en7}
\end{equation}
holds for all $u\in\mathscr{H}_{E}$, see Lemma \ref{lem:dipole}
below.
\begin{rem}
We illustrate this Parseval frame in section \ref{ex:tri} with a
finite-dimensional example.
\end{rem}

\subsection{Dipoles}

Let $\left(V,E,c,\mathscr{H}_{E}\right)$ be as described above, and
assume that $\left(V,E,c\right)$ is connected. Note that ``vectors''
in $\mathscr{H}_{E}$ are equivalence classes of functions on $V$
(= the vertex set in the graph $G=\left(V,E\right)$). 
\begin{lem}[\cite{JoPe10}]
\label{lem:dipole}For every pair of vertices $x,y\in V$, there
is a \uline{unique} vector $v_{xy}\in\mathscr{H}_{E}$ satisfying
\begin{equation}
\left\langle u,v_{xy}\right\rangle _{\mathscr{H}_{E}}=u\left(x\right)-u\left(y\right)\label{eq:di1}
\end{equation}
for all $u\in\mathscr{H}_{E}$.\end{lem}
\begin{proof}
Fix a pair of vertices $x,y$ as above, and pick a finite path of
edges $\left(x_{i},x_{i+1}\right)\in E$ such that $c_{x_{i},x_{i+1}}>0$,
and $x_{0}=y$, $x_{n}=x$. Then
\begin{align}
u\left(x\right)-u\left(y\right) & =\sum_{i=0}^{n-1}u\left(x_{i+1}\right)-u\left(x_{i}\right)\nonumber \\
 & =\sum_{i=0}^{n-1}\frac{1}{\sqrt{c_{x_{i}x_{i+1}}}}\sqrt{c_{x_{i}x_{i+1}}}\left(u\left(x_{i+1}\right)-u\left(x_{i}\right)\right);\label{eq:di2}
\end{align}
and, by Schwarz, we have the following estimate:
\begin{align*}
\left|u\left(x\right)-u\left(y\right)\right|^{2} & \leq\left(\sum_{i=0}^{n-1}\frac{1}{c_{x_{i}x_{i+1}}}\right)\sum_{j=0}^{n-1}c_{x_{j}x_{j+1}}\left|u\left(x_{j+1}\right)-u\left(x_{j}\right)\right|^{2}\\
 & \leq\left(\mbox{Const}_{xy}\right)\left\Vert u\right\Vert _{\mathscr{H}_{E}}^{2},
\end{align*}
valid for all $u\in\mathscr{H}_{E}$, where we used (\ref{eq:en6})
in the last step of this \emph{a priori} estimate. But this states
that the linear functional:
\begin{equation}
L_{xy}:\mathscr{H}_{E}\ni u\longmapsto u\left(x\right)-u\left(y\right)\label{eq:di3}
\end{equation}
is continuous on $\mathscr{H}_{E}$ w.r.t. the norm $\left\Vert \cdot\right\Vert _{\mathscr{H}_{E}}$.
Hence existence and uniqueness for $v_{xy}\in\mathscr{H}_{E}$ follows
from Riesz' theorem. We get $\exists!v_{xy}\in\mathscr{H}_{E}$ s.t.
\[
L_{xy}\left(u\right)=\left\langle v_{xy},u\right\rangle _{\mathscr{H}_{E}},\;\forall u\in\mathscr{H}_{E}.
\]
\end{proof}
\begin{rem}
\label{rem:harm}Let $x,y\in V$ be as above, and let $v_{xy}\in\mathscr{H}_{E}$
be the dipole. One checks, using (\ref{eq:cond})-(\ref{eq:lap})
that
\begin{equation}
\Delta v_{xy}=\delta_{x}-\delta_{y}.\label{eq:di4}
\end{equation}
But this equation (\ref{eq:di4}) does \uline{not} determine $v_{xy}$
uniquely. Indeed if $w$ is a function on $V$ satisfying $\Delta w=0$,
i.e., $w$ is \uline{harmonic}, then $v_{xy}+w$ also satisfies
eq. (\ref{eq:di4}). (In \cite{JoPe11a,JoPe11b} we studied when $\left(V,E,c\right)$
has non-constant harmonic functions in $\mathscr{H}_{E}$.)

The system of vectors $v_{xy}$ in (\ref{eq:en7}) indexed by pairs
of vertices carry a host of information about the given system $\left(V,E,c,\mathscr{H}_{E}\right)$,
for example the computation of \emph{\uline{resistance metric}}.\end{rem}
\begin{lem}
\label{lem:metric}When $c$ (conductance) is given and assume $\left(V,c\right)$
is connected; set
\begin{equation}
d_{c}\left(x,y\right):=\sup\left\{ \frac{1}{\left\Vert u\right\Vert _{\mathscr{H}_{E}}^{2}}\:\Big|\: u\in\mathscr{H}_{E},u\left(x\right)=1,u\left(y\right)=0\right\} \label{eq:en1-1}
\end{equation}
then $d_{c}\left(x,y\right)$ is a metric on $V$, and 
\begin{equation}
d_{c}\left(x,y\right)=\left\Vert v_{xy}\right\Vert _{\mathscr{H}_{E}}^{2}\label{eq:en1-2}
\end{equation}
where the \uline{dipole vectors}\emph{ are specified as in }(\ref{eq:en7}). \end{lem}
\begin{proof}
Consider $u\in\mathscr{H}_{E}$ as in (\ref{eq:en1-1}), i.e., $u\left(x\right)=1$,
$u\left(y\right)=0$. Using (\ref{eq:en7}) and the Schwarz inequality,
we then get
\begin{align*}
1 & =u\left(x\right)-u\left(y\right)=\left|\left\langle v_{xy},u\right\rangle _{\mathscr{H}_{E}}\right|^{2}\\
 & \leq\left\Vert v_{xy}\right\Vert _{\mathscr{H}_{E}}^{2}\left\Vert u\right\Vert _{\mathscr{H}_{E}}^{2}.\;\left(\mbox{by Schwarz}\right)
\end{align*}
Since we know the optimizing vectors in the Schwarz inequality, the
desired formula (\ref{eq:en1-2}) now follows from (\ref{eq:en1-1})
and (\ref{eq:en7}). 

But (\ref{eq:en1-1}) is known to yield a metric and the lemma follows.
\end{proof}
The next lemma offers a lower bound for the resistance metric between
any two vertices when $\left(V,E,c\right)$ is given: Given any two
vertices $x$ and $y$, we prove the following estimate: $\mbox{dist}_{c}\left(x,y\right)\geq$
sum of dissipation along any path of edges from $x$ to $y$.
\begin{lem}
Let $G=\left(V,E,c\right)$ be as before. For all finite paths $x_{0}:=x\rightarrow\left(e_{i}\right)\rightarrow x_{n}:=y$,
we have 
\begin{equation}
\mbox{dist}_{c}\left(x,y\right)\geq\underset{\mbox{dissipation}}{\underbrace{\sum_{i=0}^{n-1}\mbox{Res}_{x_{i}x_{i+1}}\left|I\left(v_{xy}\right)_{i,i+1}\right|^{2}}};\label{eq:en3-1}
\end{equation}
where $\mbox{Res}=\frac{1}{c}$ denotes the resistance. \end{lem}
\begin{proof}
In general there are many paths from $x$ to $y$ when $x$ and $y$
are fixed vertices. See Fig \ref{fig:pth}. 
\begin{align*}
\left\Vert v_{xy}\right\Vert _{\mathscr{H}_{E}}^{2} & =\mbox{dist}_{c}\left(x,y\right)=\sum_{e\in E^{\left(dir\right)}}\left|\left\langle w_{e},v_{xy}\right\rangle _{\mathscr{H}_{E}}\right|^{2}\quad\left(w_{e}:=\sqrt{c_{e}}v_{e}\right)\\
 & =\sum_{e\in E^{\left(dir\right)}}c_{c}\left|\left\langle v_{e},v_{xy}\right\rangle _{\mathscr{H}_{E}}\right|^{2}\\
 & \geq\sum_{i=0}^{n-1}c_{i,i+1}\left|v_{xy}\left(x_{i}\right)-v_{xy}\left(x_{i+1}\right)\right|^{2}\\
 & =\sum_{i=0}^{n-1}\mbox{Res}_{x_{i}x_{i+1}}\left|I\left(v_{xy}\right)_{x_{i}x_{i+1}}\right|^{2}.
\end{align*}

\end{proof}
\begin{figure}[H]
\includegraphics[scale=0.5]{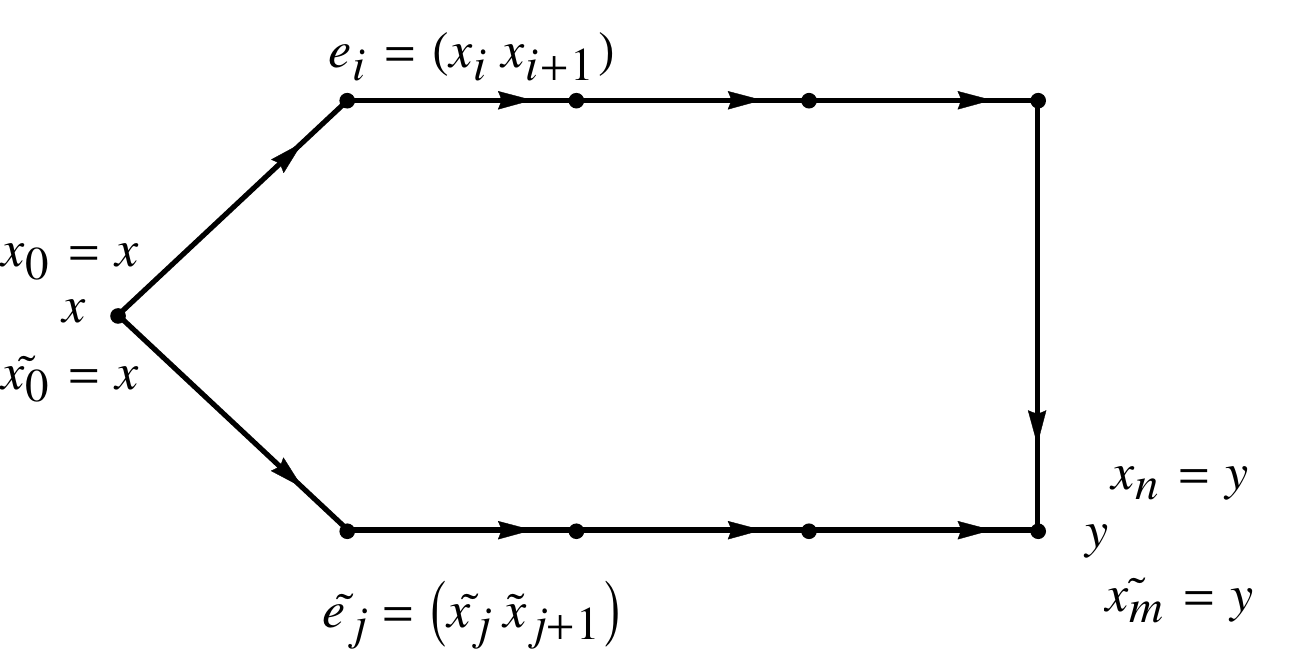}
\[
v_{xy}=\sum_{i=0}^{n-1}v_{i,i+1}=\sum_{j=0}^{m-1}v_{j,j+1}
\]

\protect\caption{\label{fig:pth}Two finite paths connecting $x$ and $y$, where $e_{i}=\left(x_{i}x_{i+1}\right),\widetilde{e_{j}}=\left(\widetilde{x}_{j}\widetilde{x}_{j+1}\right)\in E$.}

\end{figure}

Now pick an orientation for each edge, and denote by $E^{\left(ori\right)}$
the set of oriented edges. 
\begin{thm}
\label{thm:eframe}Let $\left(V,E,c,E^{\left(ori\right)}\right)$
and $\mathscr{H}_{E}$ be as above; then the system of vectors
\begin{equation}
w_{xy}:=\sqrt{c_{xy}}v_{xy},\;\mbox{indexed by \ensuremath{\left(xy\right)\in E^{\left(ori\right)}}}\label{eq:en8}
\end{equation}
is a Parseval frame for $\mathscr{H}_{E}$.\end{thm}
\begin{proof}
We will show that (\ref{eq:en1}) holds for constants $b_{1}=b_{2}=1$
for the vectors $\left(w_{\left(xy\right)}\right)_{\left(xy\right)\in E^{\left(ori\right)}}$,
see (\ref{eq:en7})-(\ref{eq:en8}).

Indeed we have for $u\in\mathscr{H}_{E}$:
\begin{eqnarray*}
\left\Vert u\right\Vert _{\mathscr{H}_{E}}^{2} & \underset{\left(\text{by \ensuremath{\left(\ref{eq:en6}\right)}}\right)}{=} & \sum_{\left(xy\right)\in E^{\left(ori\right)}}c_{xy}\left|u\left(x\right)-u\left(y\right)\right|^{2}\\
 & \underset{\left(\text{by \ensuremath{\left(\ref{eq:en7}\right)}}\right)}{=} & \sum_{\left(xy\right)\in E^{\left(ori\right)}}c_{xy}\left|\left\langle v_{xy},u\right\rangle _{\mathscr{H}_{E}}\right|^{2}\\
 & = & \sum_{\left(xy\right)\in E^{\left(ori\right)}}\left|\left\langle \sqrt{c_{xy}}v_{xy},u\right\rangle _{\mathscr{H}_{E}}\right|^{2}\\
 & \underset{\left(\text{by \ensuremath{\left(\ref{eq:en8}\right)}}\right)}{=} & \sum_{\left(xy\right)\in E^{\left(ori\right)}}\left|\left\langle w_{xy},u\right\rangle _{\mathscr{H}_{E}}\right|^{2}
\end{eqnarray*}
which is the desired conclusion.\end{proof}
\begin{rem}
While the vectors $w_{xy}:=\sqrt{c_{xy}}v_{xy}$, $\left(xy\right)\in E^{\left(ori\right)}$,
form a Parseval frame in $\mathscr{H}_{E}$ in the general case, typically
this frame is not an orthogonal basis (ONB) in $\mathscr{H}_{E}$;
although it is in Example \ref{ex:1d} below.

To see when our Parseval frames are in fact ONBs, we use the following:\end{rem}
\begin{lem}
Let $\left\{ w_{j}\right\} _{j\in J}$ be a Parseval frame in a Hilbert
space $\mathscr{H}$, then $\left\Vert w_{j}\right\Vert _{\mathscr{H}_{E}}\leq1$,
and it is an ONB in $\mathscr{H}$ if and only if $\left\Vert w_{j}\right\Vert _{\mathscr{H}}=1$
for all $j\in J$.\end{lem}
\begin{proof}
Follows from an easy application of 
\begin{equation}
\left\Vert u\right\Vert _{\mathscr{H}}^{2}=\sum_{j\in J}\left|\left\langle w_{j},u\right\rangle _{\mathscr{H}}\right|^{2},\; u\in\mathscr{H}.\label{eq:en2-1}
\end{equation}
Plug in $w_{j_{0}}$ for $u$ in (\ref{eq:en2-1}).\end{proof}
\begin{rem}
Frames in $\mathscr{H}_{E}$ consisting of our system (\ref{eq:en8})
are not ONBs when resisters are configured in non-linear systems of
vertices, for example, resisters in parallel. See Fig \ref{fig:frame},
and Example \ref{ex:tri}.

\begin{figure}[H]
\begin{tabular}[t]{>{\centering}p{0.45\textwidth}>{\centering}p{0.45\textwidth}}
\includegraphics[scale=0.5]{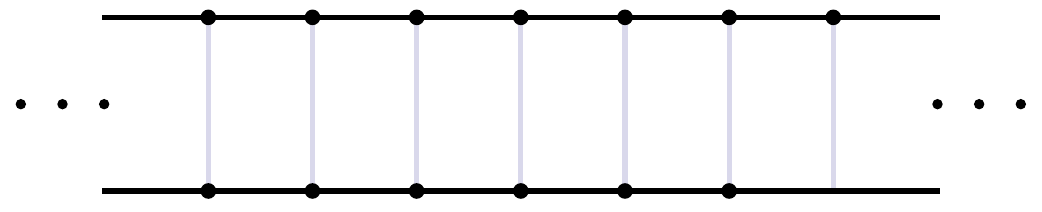}\\
 & \includegraphics[scale=0.45]{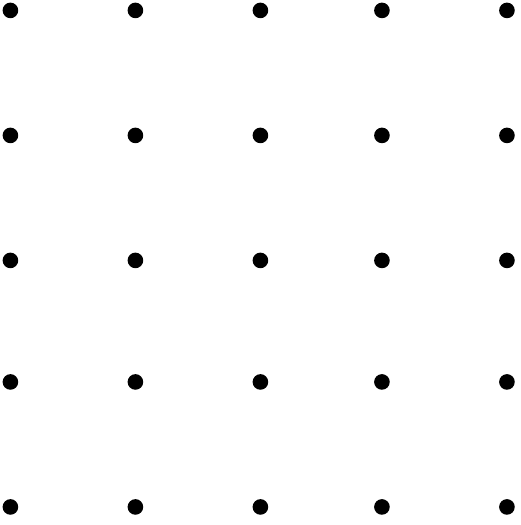}\\
\tabularnewline
$V=\mbox{Band}$ & $V=\mathbb{Z}^{2}$\tabularnewline
\end{tabular}

\protect\caption{\label{fig:frame}non-linear system of vertices}

\end{figure}

In these examples one checks that 
\begin{equation}
1>\left\Vert w_{xy}\right\Vert _{\mathscr{H}_{E}}^{2}=c_{xy}\left\Vert v_{xy}\right\Vert _{\mathscr{H}_{E}}^{2}=c_{xy}\left(v_{xy}\left(x\right)-v_{xy}\left(y\right)\right).\label{eq:en2-2}
\end{equation}
That is the current flowing through each edge $e=\left(x,y\right)\in E$
is $<1$; or equivalently the voltage-drop across $e$ is $<$ resistance
\[
v_{xy}\left(x\right)-v_{xy}\left(y\right)<\frac{1}{c_{cy}}=\mbox{resistance.}
\]
\end{rem}
\begin{defn}
\label{def:ori}Let $V,E,c,\mathscr{H}_{E}$ be as above and for $u\in\mathscr{H}_{E}$,
set 
\begin{equation}
I\left(u\right)_{\left(xy\right)}:=c_{xy}\left(u\left(x\right)-u\left(y\right)\right)\;\mbox{for \ensuremath{\left(x,y\right)\in E}}.\label{eq:en9}
\end{equation}
By Ohm's law, the function $I\left(u\right)_{\left(xy\right)}$ in
(\ref{eq:en9}) represents the current in a network. 

A choice of \emph{\uline{orientation}} may be assigned as follows
(three different ways):
\begin{enumerate}
\item The orientation of every $\left(xy\right)\in E$ may be chosen arbitrarily.
\item The orientation may be suggested by geometry; for example in a binary
tree, as shown in Fig \ref{fig:btree} below. 
\item \label{enu:ori3}Or the orientation may be assigned by the experiment
of inserting one Amp at a vertex, say $o\in V$, and extracting one
Amp at a distinct vertex, say $x_{dist}\in V$. We then say that an
edge $\left(xy\right)\in E$ is positively oriented if $I\left(u\right)_{xy}>0$
where $I\left(u\right)_{xy}$ is the induced current; see (\ref{eq:en9}).
See Fig \ref{fig:etree} below.
\end{enumerate}
\end{defn}
\begin{figure}[H]
\includegraphics[scale=0.6]{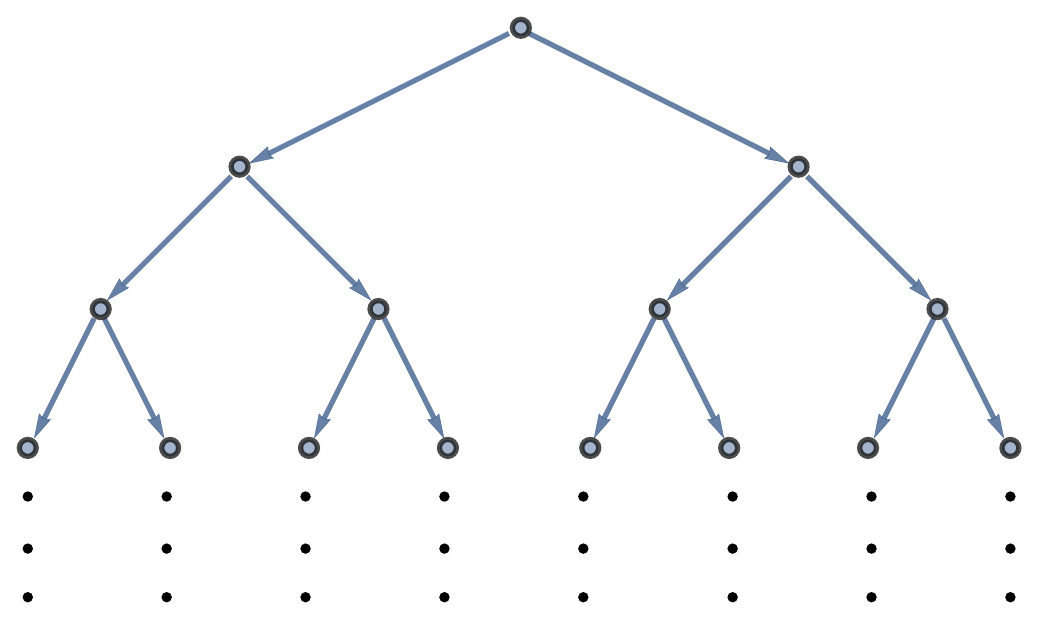}

\protect\caption{\label{fig:btree}Geometric orientation. (See also Fig \ref{fig:bt}
below. )}

\end{figure}

\begin{figure}[H]
\includegraphics[scale=0.6]{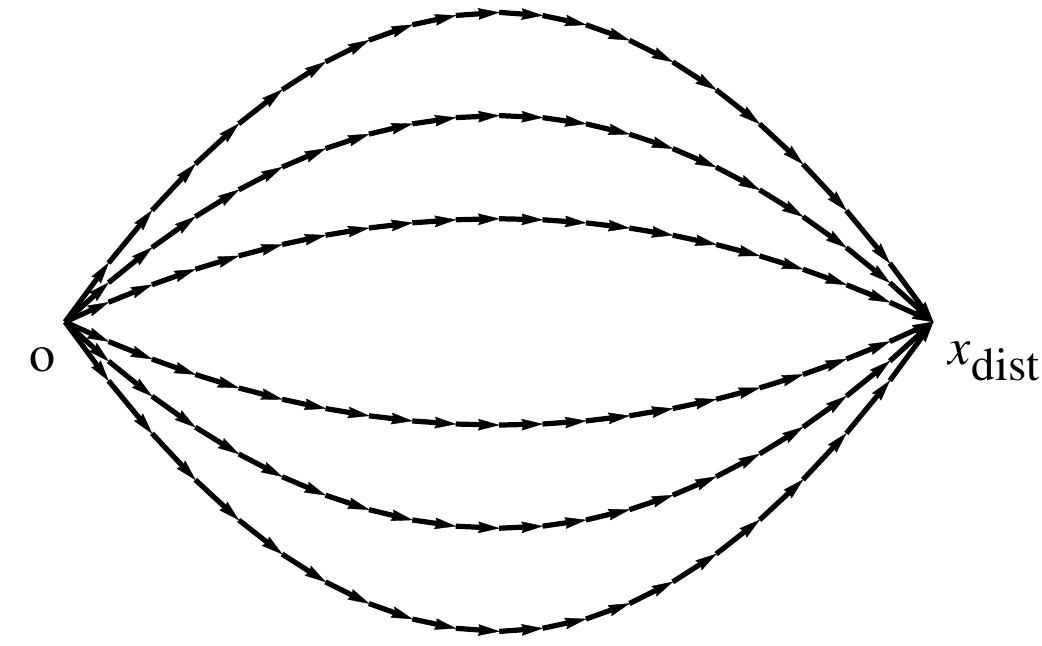}

\protect\caption{\label{fig:etree}Electrically induced orientation.}
\end{figure}

\begin{cor}
Let $\left(V,E,c,\mathscr{H}_{E}\right)$ be as above, and let $E^{\left(ori\right)}$
be assigned as in (\ref{enu:ori3}) of Definition \ref{def:ori}.
Then every $u\in\mathscr{H}_{E}$ has a norm-convergent representation
\begin{equation}
u=\sum_{\left(xy\right)\in E^{\left(ori\right)}}I\left(u\right)_{xy}v_{xy}.\label{eq:en10}
\end{equation}
\end{cor}
\begin{proof}
By Theorem \ref{thm:eframe} and Lemma \ref{lem:eframe}, we have
the following norm-convergent representation
\begin{equation}
u=\sum_{\left(xy\right)\in E^{\left(ori\right)}}\left\langle w_{xy},u\right\rangle _{\mathscr{H}_{E}}w_{xy},\label{eq:en11}
\end{equation}
see (\ref{eq:en4}). Now by statements (\ref{eq:en8}) and (\ref{eq:en9})
we get 
\begin{equation}
\left\langle w_{xy},u\right\rangle _{\mathscr{H}_{E}}w_{xy}=I\left(u\right)_{xy}v_{xy}.\label{eq:en12}
\end{equation}
Considering (\ref{eq:en11}) and (\ref{eq:en12}), the desired conclusion
(\ref{eq:en10}) then follows.
\end{proof}

\section{\label{sec:lemmas}Lemmas}

Starting with a given network $(V,E,c)$, we introduce functions on
the vertices $V$, voltage, dipoles, and point-masses; and on the
edges $E$, conductance, and current. We introduce a system of operators
which will be needed throughout, the graph-Laplacian $\Delta$, and
the transition operator $P$. We show that there are two Hilbert spaces
serving different purposes, $l^{2}(V)$, and the energy Hilbert space
$\mathscr{H}_{E}$; the latter depending on choice of conductance
function $c$.

Lemma \ref{lem:Delta} below summarizes the key properties of $\Delta$
as an operator, both in $l^{2}(V)$ and in $\mathscr{H}_{E}$. The
metric properties of networks $(V,E,c)$ depend on $\mathscr{H}_{E}$
(Lemma \ref{lem:Nc}), and not on $l^{2}(V)$.

Recall that the graph-Laplacian $\Delta$ is automatically essentially
selfadjoint as a densely defined operator in $l^{2}(V)$, but not
as a $\mathscr{H}_{E}$ operator \cite{Jor08,JoPe11b}. In section
\ref{sec:eg}  we compute examples where $(\Delta,\mathscr{H}_{E})$
has deficiency indices $(m,m)$, $m>0$. These results make use of
an associated reversible random walk, as well as the transition operator
$P$.

Let $\left(V,E,c\right)$ be as above; note we are assuming that $G=\left(V,E\right)$
is connected; so there is a point $o$ in $V$ such that every $x\in V$
is connected to $o$ via a finite path of edges. We will set $V':=V\backslash\left\{ o\right\} $,
and consider $l^{2}\left(V\right)$ and $l^{2}\left(V'\right)$. If
$x\in V$, we set 
\begin{equation}
\delta_{x}\left(y\right)=\begin{cases}
1 & \mbox{if }y=x\\
0 & \mbox{if }y\neq x
\end{cases}\label{eq:delx}
\end{equation}

Set $\mathscr{H}_{E}:=$ the set of all functions $u:V\rightarrow\mathbb{C}$
such that 
\begin{equation}
\left\Vert u\right\Vert _{\mathscr{H}_{E}}^{2}:=\frac{1}{2}\underset{\left(x,y\right)\in E}{\sum\sum}c_{xy}\left|u\left(x\right)-u\left(y\right)\right|^{2}<\infty,\label{eq:Enorm}
\end{equation}
and we note \cite{JoPe10} that $\mathscr{H}_{E}$ is a Hilbert space.
Moreover for all $x,y\in V$, there is a real-valued solution $v_{xy}\in\mathscr{H}_{E}$
to the equation
\begin{equation}
\Delta v_{x,y}=\delta_{x}-\delta_{y}.\label{eq:fsoln}
\end{equation}
If $y=o$, we set $v_{x}:=v_{x,o}$, and note 
\begin{equation}
\Delta v_{x}=\delta_{x}-\delta_{o}.\label{eq:fsoln1}
\end{equation}
In this case, we assume that $v_{x}$ is defined only for $x\in V'$.
\begin{defn}
\label{def:D}Let $\left(V,E,c,o,\Delta,\left\{ v_{x}\right\} _{x\in V'}\right)$
be as above, and set 
\begin{align}
\mathscr{D}_{2} & :=span\left\{ \delta_{x}\:\big|\: x\in V\right\} ,\;\mbox{and}\label{eq:D1}\\
\mathscr{D}_{E} & :=span\left\{ v_{x}\:\big|\: x\in V'\right\} ,\label{eq:D2}
\end{align}
where by ``span'' we mean of all finite linear combinations.\end{defn}
\begin{lem}
\label{lem:Delta}The following hold:
\begin{enumerate}
\item $\left\langle \Delta u,v\right\rangle _{l^{2}}=\left\langle u,\Delta v\right\rangle _{l^{2}}$,
$\forall u,v\in\mathscr{D}_{2}$;
\item \label{enu:D2}$\left\langle \Delta u,v\right\rangle _{\mathscr{H}_{E}}=\left\langle u,\Delta v\right\rangle _{\mathscr{H}_{E}},$
$\forall u,v\in\mathscr{D}_{E}$;
\item \label{enu:D3}$\left\langle u,\Delta u\right\rangle _{l^{2}}\geq0$,
$\forall u\in\mathscr{D}_{2}$, and
\item $\left\langle u,\Delta u\right\rangle _{\mathscr{H}_{E}}\geq0$, $\forall u\in\mathscr{D}_{E}$;
where for $u,v\in\mathscr{H}_{E}$ we set

\begin{equation}
\left\langle u,v\right\rangle _{\mathscr{H}_{E}}=\frac{1}{2}\underset{\left(x,y\right)\in E}{\sum\sum}c_{xy}\overline{\left(u\left(x\right)-u\left(y\right)\right)}\left(v\left(x\right)-v\left(y\right)\right).\label{eq:HEinner}
\end{equation}

\end{enumerate}

Moreover, we have
\begin{enumerate}[resume]
\item $\left\langle v_{x,y},u\right\rangle _{\mathscr{H}_{E}}=u\left(x\right)-u\left(y\right)$,
$\forall x,y\in V$. 
\end{enumerate}

Finally,
\begin{enumerate}[resume]
\item 
\[
\delta_{x}\left(\cdot\right)=c\left(x\right)v_{x}\left(\cdot\right)-\sum_{y\sim x}c_{xy}v_{y}\left(\cdot\right),\;\forall x\in V'.
\]

\end{enumerate}
\end{lem}
\begin{proof}
~

\uline{Proof of \mbox{(\ref{enu:D2})}} We have $\left\langle \Delta u,v\right\rangle _{\mathscr{H}_{E}}=\left\langle u,\Delta v\right\rangle _{\mathscr{H}_{E}}$,
for all $u,v\in\mathscr{D}_{E}$. Set $v_{x}:=v_{xo}$ where $o$
is a fixed base-point in $V$, $V':=V\backslash\left\{ o\right\} $,
so $\Delta v_{x}=\delta_{x}-\delta_{o}$, $x\in V'$. Set $u=\sum_{x\in V'}\xi_{x}v_{x}$,
$v=\sum_{x\in V'}\eta_{x}v_{x}$; where the summations are finite
by convention; then 
\begin{eqnarray*}
\left\langle \Delta u,v\right\rangle _{\mathscr{H}_{E}} & = & \sum_{V'}\sum_{V'}\overline{\xi_{x}}\eta_{y}\left\langle \delta_{x}-\delta_{o},v_{y}\right\rangle _{\mathscr{H}_{E}}\\
 & = & \sum_{V'}\sum_{V'}\overline{\xi_{x}}\eta_{y}\big((\delta_{x}\left(y\right)-\underset{{\scriptscriptstyle =0}}{\underbrace{\delta_{x}\left(o\right)}})-(\underset{{\scriptscriptstyle =0}}{\underbrace{\delta_{o}\left(y\right)}}-\underset{{\scriptscriptstyle =1}}{\underbrace{\delta_{o}\left(o\right)}})\big)\\
 & = & \sum_{V'}\sum_{V'}\overline{\xi_{x}}\eta_{y}\left(\delta_{xy}+1\right)\\
 & = & \sum_{V'}\overline{\xi_{x}}\eta_{y}+\left(\overline{\sum_{V'}\xi_{x}}\right)\left(\sum_{V'}\eta_{y}\right)\\
 & = & \left\langle u,\Delta v\right\rangle _{\mathscr{H}_{E}}.\;(\mbox{by symmetry})
\end{eqnarray*}

For the remaining, see \cite{JoPe10,JoPe11a}.\end{proof}
\begin{lem}
\label{lem:Nc}Let $\left(V,E,c,o\right)$ be as above; then the function
\begin{equation}
N_{c}\left(x,y\right):=\left\Vert v_{x}-v_{y}\right\Vert _{\mathscr{H}_{E}}^{2}\label{eq:Nc}
\end{equation}
is conditionally negative definite, i.e., for all finite system $\left\{ \xi_{x}\right\} \subseteq\mathbb{C}$
such that $\sum_{x\in V'}\xi_{x}=0$, we have 
\begin{equation}
\sum_{x}\sum_{y}\overline{\xi_{x}}\xi_{y}N_{c}\left(x,y\right)\leq0.\label{eq:Nc1}
\end{equation}
\end{lem}
\begin{proof}
Compute the LHS in (\ref{eq:Nc1}) as follows. If $\sum\xi_{x}=0$,
we have 
\begin{align*}
 & \sum\overline{\xi_{x}}\xi_{y}N_{c}\left(x,y\right)\\
= & -\sum\sum\overline{\xi_{x}}\xi_{y}\left\langle v_{x},v_{y}\right\rangle _{\mathscr{H}_{E}}-\sum\sum\overline{\xi_{x}}\xi_{y}\left\langle v_{y},v_{x}\right\rangle _{\mathscr{H}_{E}}\\
= & -2\left\Vert \sum_{x}\xi_{x}v_{x}\right\Vert _{\mathscr{H}_{E}}^{2}.
\end{align*}

\end{proof}
We show also the following
\begin{lem}
\cite{JoPe10}
\begin{equation}
\left\{ u\in\mathscr{H}_{E}\:\big|\:\left\langle u,\delta_{x}\right\rangle _{\mathscr{H}_{E}}=0,\;\forall x\in V\right\} =\left\{ u\in\mathscr{H}_{E}\:\big|\:\Delta u=0\right\} .\label{eq:harm}
\end{equation}

\end{lem}
When $N$ is a fixed negative definite function, we get an associated
Hilbert space $\mathscr{H}_{N}$ by completing finitely supported
functions $\xi$ on $V$ subject to the condition $\sum_{x\in V}\xi_{x}=0$,
under the inner product 
\[
\left\Vert \xi\right\Vert _{\mathscr{H}_{N}}^{2}:=-\sum_{x}\sum_{y}\overline{\xi_{x}}\xi_{y}N\left(x,y\right)
\]
and quotienting out with 
\[
\sum_{x}\sum_{y}\overline{\xi_{x}}\xi_{y}N\left(x,y\right)=0.
\]

\begin{lem}
Assume $\left(V,E,c\right)$ is connected. If a negative definite
function $N$ on $V\times V$ satisfies $N=N_{c}$, then 
\begin{equation}
\mathscr{H}_{N_{c}}=\mathscr{H}_{E}\label{eq:N2-1}
\end{equation}
where $\mathscr{H}_{E}$ is the energy Hilbert space from \secref{eframe}
defined on the prescribed functions.\end{lem}
\begin{proof}
By Lemma \ref{lem:Nc} we have 
\begin{equation}
\left\Vert \xi\right\Vert _{\mathscr{H}_{N_{c}}}^{2}=\left\Vert \sum_{x\in V'}\xi_{x}v_{x}\right\Vert _{\mathscr{H}_{E}}^{2}\label{eq:N2-2}
\end{equation}
where $v_{x}=v_{ox}$ is a system of dipoles corresponding to a fixed
base point $o\in V$, and $V'=V\backslash\left\{ o\right\} $, and
\begin{equation}
\left\langle v_{x},u\right\rangle _{\mathscr{H}_{E}}=u\left(x\right)-u\left(o\right),\;\forall u\in\mathscr{H}_{E}.\label{eq:N2-3}
\end{equation}
Hence, we need only prove that all the finite summations $\sum_{x\in V'}\xi_{x}v_{x}$
subjecting to $\sum_{x\in V'}\xi_{x}=0$ are dense in $\mathscr{H}_{E}$.
But if $u\in\mathscr{H}_{E}$, $u\in\left\{ \sum_{x\in V'}\xi_{x}v_{x}\:\big|\:\sum_{x\in V'}\xi_{x}=0\right\} ^{\perp}$
(the orthogonal-complement) then 
\[
\left\langle v_{x},u\right\rangle _{\mathscr{H}_{E}}-\left\langle v_{y},u\right\rangle _{\mathscr{H}_{E}}=0
\]
for all $x,y\in V'$. Hence by (\ref{eq:N2-3}) we get $u\left(x\right)=u\left(y\right)$
for all pairs $x,y\in V'$; and $u\left(x\right)=u\left(o\right)$,
$x\in V'$. Since $\left(V,E,c\right)$ is connected, it follows that
$u$ is constant. But with the normalization $v_{x}\left(o\right)=0$
$\left(x\in V'\right)$, we conclude that $u$ must be zero.\end{proof}
\begin{lem}
\label{lem:deltaD}Let $\left(V,E,c\right)$ be a connected network,
and let $\mathscr{H}_{E}$ be the energy Hilbert space; then, for
all $f\in\mathscr{H}_{E}$ and $x\in V$, we have 
\begin{equation}
\left\langle \delta_{x},f\right\rangle _{\mathscr{H}_{E}}=\left(\Delta f\right)\left(x\right).\label{eq:N3-1}
\end{equation}
\end{lem}
\begin{proof}
We compute $\mbox{LHS}_{\left(\ref{eq:N3-1}\right)}$ with the use
of eq. (\ref{eq:HEinner}) in Lemma \ref{lem:Delta}. Indeed 
\begin{eqnarray*}
\left\langle \delta_{x},f\right\rangle _{\mathscr{H}_{E}} & \underset{\left(\ref{eq:HEinner}\right)}{=} & \frac{1}{2}\underset{\left(st\right)\in E}{\sum\sum}c_{st}\left(\delta_{x}\left(s\right)-\delta_{x}\left(t\right)\right)\left(f\left(s\right)-f\left(t\right)\right)\\
 & = & \sum_{t\sim x}c_{xt}\left(f\left(x\right)-f\left(t\right)\right)=\left(\Delta f\right)\left(x\right)
\end{eqnarray*}
where we used (\ref{eq:lap}) in Definition \ref{def:lap} in the
last step.\end{proof}
\begin{cor}
\label{cor:deltain}Let $\left(V,E,c\right)$ be as above, then
\begin{equation}
\left\langle \delta_{x},\delta_{y}\right\rangle _{\mathscr{H}_{E}}=\begin{cases}
\widetilde{c}\left(x\right)=\sum_{t\sim x}c_{xt} & \mbox{if \ensuremath{y=x}}\\
-c_{xy} & \mbox{if \ensuremath{\left(xy\right)\in E}}\\
0 & \mbox{if \ensuremath{\left(xy\right)\in E}and \ensuremath{x\neq y}}
\end{cases}\label{eq:N3-2}
\end{equation}
\end{cor}
\begin{proof}
Immediate from the lemma.\end{proof}
\begin{cor}
Let $\left(V,E,c,\Delta\right)$ be as above; then
\begin{equation}
\mathscr{H}_{E}\ominus\left\{ \delta_{x}\:|\: x\in V\right\} =\left\{ u\in\mathscr{H}_{E}\:|\:\Delta u=0\right\} .\label{eq:N3-3}
\end{equation}
\end{cor}
\begin{proof}
This is immediate from (\ref{eq:N3-1}) in Lemma \ref{lem:deltaD}.
(Note that ``$\ominus$'' in (\ref{eq:N3-3}) means ortho-complement.)\end{proof}
\begin{cor}
Let $\left(V,E,c,\Delta\right)$ be as above; then for every $x\in V$,
we have:
\begin{equation}
\sum_{y\sim x}c_{xy}v_{xy}=\delta_{x}.\label{eq:N3-4}
\end{equation}
\end{cor}
\begin{proof}
Immediate from (\ref{eq:N3-1}) in Lemma \ref{lem:deltaD}.
\end{proof}

\section{\label{sec:hps}The Hilbert Spaces $\mathscr{H}_{E}$ and $l^{2}\left(V'\right)$,
and Operators Between Them}

The purpose of this section is to prepare for the two results (sections
\ref{sec:hps} and \ref{sec:Fext}) on factorization to follow.
\begin{defn}
Let $\left(V,E,c,o,\left\{ v_{x}\right\} _{x\in V'}\right)$ be specified
as in sections \ref{sec:setting}-\ref{sec:lemmas}, where $c$ is
a fixed conductance function. Set 
\begin{align}
\mathscr{D}_{l^{2}}' & :=\mbox{all finitely supported functions \ensuremath{\xi}on \ensuremath{V'}such that }\ensuremath{\sum_{x}\xi_{x}=0}.\label{eq:Dl}
\end{align}

\end{defn}
Then assuming that $V$ is infinite, we conclude that $\mathscr{D}'_{l^{2}}$
is dense in $l^{2}\left(V'\right)$. 
\begin{lem}
For $\left(\xi_{x}\right)\in\mathscr{D}_{l^{2}}'$, set 
\begin{equation}
K\left(\xi\right):=\sum_{x\in V'}\xi_{x}v_{x}\in\mathscr{H}_{E}.\label{eq:K}
\end{equation}
Then $K\left(=K_{c}\right)$ is a densely defined, and closable operator
\[
K:l^{2}\left(V'\right)\longrightarrow\mathscr{H}_{E}
\]
with domain $\mathscr{D}'_{l^{2}}$. \end{lem}
\begin{proof}
We must prove that the norm-closure of the graph of $K$ in $l^{2}\times\mathscr{H}_{E}$
is again the graph of a linear operator; equivalently, if $\lim_{n\rightarrow\infty}\left\Vert \xi^{\left(n\right)}\right\Vert _{l^{2}}=0$,
$\xi^{\left(n\right)}\in\mathscr{D}_{l^{2}}$; and if $\exists u\in\mathscr{H}_{E}$
such that
\begin{equation}
\lim_{n\rightarrow\infty}\left\Vert K\left(\xi^{\left(n\right)}\right)-u\right\Vert _{\mathscr{H}_{E}}=0,\label{eq:H1}
\end{equation}
then $u=0$ in $\mathscr{H}_{E}$. 

We prove this by establishing a formula for an adjoint operator, 
\[
K^{*}:\mathscr{H}_{E}\longrightarrow l^{2}\left(V'\right)
\]
having as its domain
\begin{equation}
\left\{ \sum_{x}\xi_{x}v_{x}\:\big|\:\mbox{finite sums, \ensuremath{\xi_{x}\in\mathbb{C}}, s.t. \ensuremath{\sum_{x}\xi_{x}=0}}\right\} .\label{eq:domKadj}
\end{equation}
Setting 
\begin{gather}
K^{*}\left(\sum_{x}\xi_{x}v_{x}\right)=\left(\zeta_{x}\right),\mbox{ where }\nonumber \\
\zeta_{x}=\sum_{y}\left\langle v_{x},v_{y}\right\rangle _{\mathscr{H}_{E}}\xi_{y}\label{eq:Kadj}
\end{gather}
on the space in (\ref{eq:domKadj}), we show that this is a well-defined,
densely defined, linear operator, and that
\begin{equation}
\left\langle K^{*}\left(\sum_{x}\xi_{x}v_{x}\right),\eta\right\rangle _{l^{2}}=\left\langle \sum_{x}\xi_{x}v_{x},K\eta\right\rangle _{\mathscr{H}_{E}}\label{eq:H2}
\end{equation}
holds for all $\eta\in\mathscr{D}'_{l^{2}}$. This shows that a well-defined
adjoint $K^{*}$ operator exists (by (\ref{eq:Kadj})), and that therefore
the implication in (\ref{eq:H1}) is valid. 

Now let $\xi,\eta\in\mathscr{D}_{l^{2}}'$ as in (\ref{eq:Dl}); then
\begin{eqnarray*}
\left(\mbox{LHS}\right)_{\left(\ref{eq:H2}\right)} & = & \left\langle \sum_{y}\left\langle v_{x},v_{y}\right\rangle _{\mathscr{H}_{E}}\xi_{y},\eta\right\rangle _{l^{2}}\\
 & = & \sum_{x}\sum_{y}\overline{\xi_{y}}\left\langle v_{y},v_{x}\right\rangle _{\mathscr{H}_{E}}\eta_{x}\\
 & = & \left\langle \sum_{y}\xi_{y}v_{y},\sum_{x}\eta_{x}v_{x}\right\rangle _{\mathscr{H}_{E}}\\
 & \underset{\left(\text{by \ensuremath{\left(\ref{eq:K}\right)}}\right)}{=} & \left\langle \sum_{y}\xi_{y}v_{y},K\eta\right\rangle _{\mathscr{H}_{E}}=\left(\mbox{RHS}\right)_{\left(\ref{eq:H2}\right)}
\end{eqnarray*}

\end{proof}

\section{\label{sec:Fext}The Friedrichs Extension}

Below we fix a conductance function $c$ which turns the system $\left(V,E,c\right)$
into a connected network (see sect \ref{sec:setting}), and we will
study the $c$-graph Laplacian $\Delta$ in $\mathscr{H}_{E}$, the
energy Hilbert space.

Notice that $\Delta$ will then be densely defined in $\mathscr{H}_{E}$,
see Definition \ref{def:D} and Lemma \ref{lem:Delta}. Below we study
the Friedrichs extension of $\Delta$ when it is defined on its natural
dense domain $\mathscr{D}_{E}$ in $\mathscr{H}_{E}$.

Let $\left(V,E,c,o,\left\{ v_{x}\right\} ,\Delta\right)$ be as above,
i.e.,
\begin{itemize}[itemsep=1em]
\item $G\;\begin{cases}
V=\mbox{set of vertices, assumed countable infinite }\aleph_{0}\\
E=\mbox{edges, \ensuremath{V}assumed \ensuremath{E-}connected}
\end{cases}$
\item $c:E\longrightarrow\mathbb{R}_{+}\cup\left\{ 0\right\} $ a fixed
conductance function
\item $\Delta\left(:=\Delta_{c}\right)$ the graph Laplacian
\item $o\in V$ a fixed base-point, such that $\Delta v_{x}=\delta_{x}-\delta_{o}$ 
\item $V':=V\backslash\{o\}$
\item $\mathscr{H}_{E}:=\mbox{span}\left\{ v_{x}\:|\: x\in V'\right\} $ 
\end{itemize}
Recall we proved in \secref{setting} that $\Delta$ is a semibounded
Hermitian operator with dense domain $\mathscr{D}_{E}$ in $\mathscr{H}_{E}$. 

In this section, we shall be concerned with its Friedrichs extension,
now denoted $\Delta_{Fri}$; for details on the Friedrichs extension,
see e.g., \cite{DS88b,AG93}; and, in the special case of $\left(\Delta,\mathscr{H}_{E}\right)$,
see \cite{JoPe10,JoPe11b}.

In all cases, we have that $\Delta$, $\Delta_{Fri}$, and $\Delta^{*}$
as operators in $\mathscr{H}_{E}$ act on subspaces of functions on
$V$ via the following formula:
\begin{equation}
\left(\Delta u\right)\left(x\right)=\sum_{y\sim x}c_{xy}\left(u\left(x\right)-u\left(y\right)\right)\label{eq:F1}
\end{equation}
where $u$ is a function on the vertex set $V$.
\begin{lem}
\label{lem:lapd}As an operator in $\mathscr{H}_{E}$, the graph Laplacian
$\Delta$ (with domain $\mathscr{D}_{E}$) may, or may not, be essentially
selfadjoint. Its deficiency indices are $\left(m,m\right)$ where
\begin{equation}
m=\dim\left\{ u\in\mathscr{H}_{E}\:\big|\: s.t.\:\Delta u=-u\right\} .\label{eq:F2-1}
\end{equation}
\end{lem}
\begin{proof}
Recall that if $S$ is a densely defined operator in a Hilbert space
$\mathscr{H}$ s.t. 
\begin{equation}
\left\langle u,Su\right\rangle \geq0,\;\forall u\in dom\left(S\right)\label{eq:F2-2}
\end{equation}
then $S$ will automatically have indices $\left(m,m\right)$ where
$m=\dim\left(\mathscr{N}\left(S^{*}+I\right)\right)$, and where $S^{*}$
denotes the adjoint operator, i.e., 
\begin{equation}
dom\left(S^{*}\right)=\Big\{ u\in\mathscr{H}\:\big|\:\exists C<\infty\: s.t.\:\left|\left\langle u,S\varphi\right\rangle \right|\leq C\left\Vert \varphi\right\Vert ,\;\forall\varphi\in dom\left(S\right)\Big\}.\label{eq:F2-3}
\end{equation}

We may apply this to $\mathscr{H}=\mathscr{H}_{E}$, and $S:=\Delta$
on the domain $\mathscr{D}_{E}$. One checks that, if $u\in dom\left(S^{*}\right)$,
i.e., $u\in dom((\Delta|_{\mathscr{D}_{E}})^{*})$, then 
\[
\left(S^{*}u\right)\left(x\right)=\sum_{y\in E\left(x\right)}c_{xy}\left(u\left(x\right)-u\left(y\right)\right),
\]
(i.e., the pointwise action of $\Delta$ on functions) and the conclusion
in (\ref{eq:F2-1}) follows from the assertion about $\mathscr{N}\left(S^{*}+I\right)$.\end{proof}
\begin{cor}
\label{cor:P}Let $p_{xy}:=\frac{c_{xy}}{c\left(x\right)}$ be the
transition-probabilities in Remark \ref{rem:rw} (see also Fig \ref{fig:tp}),
and let 
\begin{equation}
\left(Pu\right)\left(x\right):=\sum_{y\sim x}p_{xy}u\left(y\right)\label{eq:F3-1}
\end{equation}
be the corresponding transition operator, accounting for the $p$-random
walk on $\left(V,E\right)$. Let $\left(\Delta,\mathscr{H}_{E}\right)$
be the $\mathscr{H}_{E}$-symmetric operator with domain $\mathscr{D}_{E}$,
see Definition \ref{def:D}. Then $\left(\Delta,\mathscr{H}_{E}\right)$
has deficiency indices $\left(m,m\right)$, $m>0$, if and only if
there is a function $u$ on $V$, $u\neq0$, $u\in\mathscr{H}_{E}$
satisfying 
\begin{equation}
\left(1+\frac{1}{c\left(x\right)}\right)u\left(x\right)=\left(Pu\right)\left(x\right),\;\mbox{for all \ensuremath{x\in V}.}\label{eq:F3-2}
\end{equation}
\end{cor}
\begin{proof}
Since $(\Delta|_{\mathscr{D}_{E}})^{*}$ acts pointwise on functions
on $V$; see Lemma \ref{lem:lapd}; we only need to verify that the
equation $-u=\Delta u$ translates into (\ref{eq:F3-2}), but we have:
\[
-u\left(x\right)=c\left(x\right)u\left(x\right)-\sum_{y\sim x}c_{xy}u\left(y\right)\Longleftrightarrow\left(1+\frac{1}{c\left(x\right)}\right)u\left(x\right)=\sum_{y\sim x}\frac{c_{xy}}{c\left(x\right)}u\left(y\right)
\]
which is the desired eq. (\ref{eq:F3-2}).

For $u$ from (\ref{eq:F3-2}) to be in $dom((\Delta|_{\mathscr{D}_{E}})^{*})$
we must have 
\[
\sum_{\left(xy\right)\in E}c_{xy}\left|u\left(x\right)-u\left(y\right)\right|^{2}<\infty
\]
as asserted.
\end{proof}
In the discussion below, we use that both operators $\Delta$ and
$P$ take real valued functions on $V$ to real valued functions,
and that $P$ is positive, satisfying $P\mathbf{1}=\mathbf{1}$ where
$\mathbf{1}$ is the constant function ``one'' on $V$.
\begin{lem}
If $u$ is a non-zero real valued function on $V$ satisfying (\ref{eq:F2-1}),
or equivalently (\ref{eq:F3-2}), and if $p\in V$ satisfies $u\left(p\right)\neq0$,
then there is an infinite path of edges $\left(x_{i}x_{i+1}\right)\in E$
such that $x_{0}=p$, and 
\begin{equation}
u\left(x_{k+1}\right)\geq\prod_{i=0}^{k}\left(1+\frac{1}{c\left(x_{i}\right)}\right)u\left(p\right).\label{eq:F4-1}
\end{equation}
\end{lem}
\begin{proof}
We may assume without loss of generality that $u\left(p\right)>0$.
Set $x_{0}=p$, and $x_{1}:=\arg\max\left\{ u\left(y\right)\:|\: y\sim x_{0}\right\} $,
so $u\left(x_{1}\right)=\max u\big|_{E\left(x_{0}\right)}$; then
\[
u\left(x_{1}\right)\geq\sum_{y\sim x_{0}}p_{x_{0}y}u\left(y\right)=\left(Pu\right)\left(x_{0}\right)=\left(1+\frac{1}{c\left(x_{0}\right)}\right)u\left(x_{0}\right)
\]
where we used (\ref{eq:F3-2}) in the last step.

Now for the induction: Suppose $x_{1},\ldots,x_{k}$ have been found
as specified; then set $x_{k+1}:=\arg\max\left\{ u\left(y\right)\:|\: y\sim x_{k}\right\} $,
so 
\[
u\left(x_{k+1}\right)\geq\left(Pu\right)\left(x_{k}\right)=\left(1+\frac{1}{c\left(x_{k}\right)}\right)u\left(x_{k}\right).
\]
A final iteration then yields the desired conclusion (\ref{eq:F4-1}). \end{proof}
\begin{rem}
In section \ref{ex:1d} below, we illustrate a family of systems $\left(V,E,c\right)$
where $\Delta$ in $\mathscr{H}_{E}$ has indices $\left(1,1\right)$.
In these examples, $V=\mathbb{Z}_{+}\cup\left\{ 0\right\} $, and
the edges $E$ consists of nearest neighbor links, i.e., if $x\in\mathbb{Z}_{+}$,
\[
E\left(x\right)=\left\{ x-1,x+1\right\} ,\mbox{ while }E\left(0\right)=\left\{ 1\right\} .
\]
\end{rem}
\begin{lem}
\label{lem:F1}A function $u$ on $V$ is in the domain of $\Delta_{Fri}$
(the Friedrichs extension) if and only if $u$ is in the completion
of $\mathscr{D}_{E}$ with respect to the quadratic form
\begin{gather}
\mathscr{D}_{E}\ni\varphi\longmapsto\left\langle \varphi,\Delta\varphi\right\rangle _{\mathscr{H}_{E}}\in\mathbb{R}_{+}\cup\left\{ 0\right\} ,\mbox{ and}\label{eq:F2}\\
\Delta u\in\mathscr{H}_{E}.\label{eq:F3}
\end{gather}
\end{lem}
\begin{proof}
The assertion follows from an application of the characterization
of $\Delta_{Fri}$ in \cite{DS88b,AG93} combined with the following
fact:

If $\varphi=\sum_{x\in V}\xi_{x}v_{x}$ is a finite sum with coefficient
$\left(\xi_{x}\right)$ satisfying $\sum_{x}\xi_{x}=0$, then 
\begin{equation}
\left\langle \varphi,\Delta\varphi\right\rangle _{\mathscr{H}_{E}}=\sum_{x\in V'}\left|\xi_{x}\right|^{2}.\label{eq:F5}
\end{equation}
Moreover, eq. (\ref{eq:F3}) holds if and only if 
\[
\left\Vert \Delta u\right\Vert _{\mathscr{H}_{E}}^{2}:=\frac{1}{2}\sum_{\left(x,y\right)\in E}c_{xy}\left|\left(\Delta u\right)\left(x\right)-\left(\Delta u\right)\left(y\right)\right|^{2}<\infty.
\]

We now prove formula (\ref{eq:F5}): 

Assume $\varphi=\sum_{x\in V'}\xi_{x}v_{x}$ is as stated; then 
\begin{align*}
\left\langle \varphi,\Delta\varphi\right\rangle _{\mathscr{H}_{E}} & =\left\langle \text{\ensuremath{\sum}}_{x}\xi_{x}v_{x},\Delta\big(\text{\ensuremath{\sum}}_{x}\xi_{y}v_{y}\big)\right\rangle _{\mathscr{H}_{E}}\\
 & =\left\langle \text{\ensuremath{\sum}}_{x}\xi_{x}v_{x},\text{\ensuremath{\sum}}_{y}\xi_{y}\left(\delta_{y}-\delta_{o}\right)\right\rangle _{\mathscr{H}_{E}}\\
 & =\text{\ensuremath{\sum}}_{x}\text{\ensuremath{\sum}}_{y}\overline{\xi_{x}}\xi_{y}\left\langle v_{x},\delta_{y}\right\rangle _{\mathscr{H}_{E}}\;(\mbox{since }\tiny\text{\ensuremath{\sum}}_{y}\xi_{y}=0)\\
 & =\text{\ensuremath{\sum}}_{x}\left|\xi_{x}\right|^{2}.
\end{align*}
\end{proof}
\begin{lem}
\label{lem:LL}Let $\mathscr{H}_{E}$ denote the completion of $\mathscr{D}_{E}$;
and let $\mathscr{D}_{l^{2}}'$ be the dense subspace in $l^{2}\left(V'\right)$,
given by $\sum_{x\in V'}\xi_{x}=0$, $\sum_{x}\left|\xi_{x}\right|^{2}<\infty$.
Set 
\begin{equation}
L\left(\xi_{x}\right):=\sum_{x}\xi_{x}\delta_{x};\label{eq:F6}
\end{equation}
then $L:l^{2}\left(V'\right)\longrightarrow\mathscr{H}_{E}$ is a
closable operator with dense domain $\mathscr{D}_{l^{2}}'$; and the
corresponding adjoint operator
\[
L^{*}:\mathscr{H}_{E}\longrightarrow l^{2}\left(V'\right)
\]
satisfies
\begin{equation}
L^{*}\left(\text{\ensuremath{\sum}}_{x\in V'}\xi_{x}v_{x}\right)=\xi.\label{eq:F7}
\end{equation}
\end{lem}
\begin{proof}
To prove the assertion, we must show that, if $\xi$ is a finitely
supported function on $V'$ such that $\sum_{x\in V'}\xi_{x}=0$,
then 
\begin{alignat}{1}
\left\langle L\left(\xi\right),u\right\rangle _{\mathscr{H}_{E}} & =\left\langle \xi,\eta\right\rangle _{l^{2}}\mbox{ where}\label{eq:F8}\\
u & =\sum_{y\in V'}\eta_{y}v_{y}.\label{eq:F9}
\end{alignat}
We prove (\ref{eq:F8}) as follows:
\begin{eqnarray*}
\left(\mbox{LHS}\right)_{\left(\ref{eq:F8}\right)} & = & \left\langle \Delta(\text{\ensuremath{\sum}}_{x}\xi_{x}v_{x}),\text{\ensuremath{\sum}}_{y}\eta_{y}v_{y}\right\rangle _{\mathscr{H}_{E}}\\
 & = & \left\langle \text{\ensuremath{\sum}}_{x}\xi_{x}v_{x},\Delta(\text{\ensuremath{\sum}}_{y}\eta_{y}v_{y})\right\rangle _{\mathscr{H}_{E}}\\
 & = & \text{\ensuremath{\sum}}_{x}\overline{\xi_{x}}\eta_{x}=\left(\mbox{RHS}\right)_{\left(\ref{eq:F8}\right)}
\end{eqnarray*}
where we used formula (\ref{eq:F5}) in the last step of the computation.\end{proof}
\begin{rem}
\label{rem:hsp}To understand $\Delta$ as an operator in the respective
subspaces of $\mathscr{H}_{E}$ recall that $\mathscr{H}_{E}$ contains
two systems of vectors $\left\{ \delta_{x}\right\} $ and $\left\{ v_{x}\right\} $
both indexed by $V'$. 

Neither of the two systems is orthogonal in the inner product of $\mathscr{H}_{E}$. 

We have $\left\langle v_{x},v_{y}\right\rangle _{\mathscr{H}_{E}}=v_{x}\left(y\right)=v_{y}\left(x\right)$;
recall our normalization $v_{x}\left(o\right)=0$ where $o$ is the
fixed base-point. Moreover (see Coroll \ref{cor:deltain}): 
\[
\left\langle \delta_{x},\delta_{y}\right\rangle _{\mathscr{H}_{E}}=\begin{cases}
-c_{xy} & \mbox{if \ensuremath{\left(x,y\right)\in E}}\\
c\left(x\right) & \mbox{if \ensuremath{x=y}}\\
0 & \mbox{otherwise}
\end{cases};\mbox{ and}
\]
\[
\left\langle \delta_{x},v_{y}\right\rangle _{\mathscr{H}_{E}}=\begin{cases}
\delta_{xy} & \mbox{if \ensuremath{x,y\in V'}}\\
-1 & \mbox{if \ensuremath{x=o,y\in V'}}.
\end{cases}
\]
\end{rem}
\begin{cor}
If $x\in V'$ then $\delta_{x}\in\mathscr{H}_{E}$, and 
\begin{equation}
\delta_{x}\left(\cdot\right)=c\left(x\right)v_{x}\left(\cdot\right)-\sum_{y\sim x}c_{xy}v_{y}\left(\cdot\right).\label{eq:F1-1}
\end{equation}
\end{cor}
\begin{proof}
To show this, it is enough to check equality of 
\[
\left\langle \mbox{LHS}_{\left(\ref{eq:F1-1}\right)},v_{x}\right\rangle _{\mathscr{H}_{E}}=\left\langle \mbox{RHS}_{\left(\ref{eq:F1-1}\right)},v_{x}\right\rangle _{\mathscr{H}_{E}}\;\mbox{for all \ensuremath{x\in V}};
\]
and this follows from an application of the formulas in Remark \ref{rem:hsp}
above.\end{proof}
\begin{thm}
\label{thm:Deltaf}Let $\left(V,E,c,\Delta\left(=\Delta_{c}\right),\mathscr{H}_{E},\Delta_{Fri}\right)$
be as above. Let $L$ and $L^{*}$ be the closed operators from Lemma
\ref{lem:LL}, then
\begin{enumerate}[label=(\roman{enumi}),ref=\roman{enumi}]
\item \label{enu:LL1}$LL^{*}$ is selfadjoint, and
\item \label{enu:.LL2}$LL^{*}=\Delta_{Fri}$.
\end{enumerate}
\end{thm}
\begin{proof}
Conclusion (\ref{enu:LL1}) follows for every closed operator $L$
with dense domain, we have that $LL^{*}$ is selfadjoint. To prove
(\ref{enu:.LL2}), we must verify that
\begin{equation}
LL^{*}u=\Delta u\;\mbox{for all \ensuremath{u\in\mathscr{D}_{E}'}.}\label{eq:F10}
\end{equation}
\uline{Proof of \mbox{(\ref{eq:F10})}:} Let $u=\sum_{x\in V'}\xi_{x}v_{x}$
be a finite sum such that $\sum_{x}\xi_{x}=0$; then
\begin{eqnarray*}
LL^{*}u & \underset{\left(\text{by \ensuremath{\left(\ref{eq:F7}\right)}}\right)}{=} & L\xi\\
 & \underset{\left(\text{by \ensuremath{\left(\ref{eq:F6}\right)}}\right)}{=} & \sum_{x\in V'}\xi_{x}\delta_{x}\\
 & \underset{\left(\text{since }\sum\xi_{x}=0\right)}{=} & \sum_{x\in V'}\xi_{x}\left(\delta_{x}-\delta_{o}\right)\\
 & \underset{\left(\text{by \ensuremath{\left(\ref{eq:F6}\right)}}\right)}{=} & \sum_{x\in V'}\xi_{x}\Delta v_{x}\\
 & \underset{\text{since finite sum}}{=} & \Delta\left(\sum_{x}\xi_{x}v_{x}\right)\\
 & = & \Delta u=\left(\mbox{RHS}\right)_{\left(\ref{eq:F10}\right)}.
\end{eqnarray*}
\end{proof}
\begin{cor}
We have the Greens-Gauss identity:
\begin{equation}
\left(\Delta_{y}\left(\left\langle v_{x},v_{y}\right\rangle \right)\right)\left(z\right)=\delta_{xz},\;\forall x,y,z\in V'.\label{eq:F11}
\end{equation}

\end{cor}
\uline{Notation:} The inner products $\left\langle v_{x},v_{y}\right\rangle :=\left\langle v_{x},v_{y}\right\rangle _{\mathscr{H}_{E}}$
constitute the Gramian of the frame Theorem \ref{thm:eframe}.
\begin{proof}
~
\begin{eqnarray*}
\left(\mbox{LHS}\right)_{\left(\ref{eq:F11}\right)} & = & \sum_{w\sim z}c_{zw}\left(\left\langle v_{x},v_{z}\right\rangle _{\mathscr{H}_{E}}-\left\langle v_{x},v_{w}\right\rangle _{\mathscr{H}_{E}}\right)\\
 & = & \sum_{w\sim z}c_{zw}\left(v_{x}\left(z\right)-v_{x}\left(w\right)\right)\\
 & = & \left(\Delta v_{x}\right)\left(z\right)\\
 & = & \left(\delta_{x}-\delta_{o}\right)\left(z\right)=\delta_{xz}=\left(\mbox{RHS}\right)_{\left(\ref{eq:F11}\right)}
\end{eqnarray*}
where we used that $z\in V'=V\backslash\left\{ o\right\} $ in the
last step of the verification.
\end{proof}

\subsection{The operator $P$ versus $\Delta$}

Let $\left(V,E,c\right)$ be a network with vertices $V$, edges $E$,
and conductance function $c:E\rightarrow\mathbb{R}_{+}\cup\left\{ 0\right\} $,
setting 
\begin{equation}
\begin{split}\widetilde{c}\left(x\right):= & \sum_{y\sim x}c_{xy},\; p_{xy}:=\frac{c_{xy}}{\widetilde{c}\left(x\right)};\mbox{ and}\\
\left(\Delta u\right)\left(x\right) & =\sum_{y\sim x}c_{xy}\left(u\left(x\right)-u\left(y\right)\right),\\
\left(Pu\right)\left(x\right) & =\sum_{y\sim x}p_{xy}u\left(y\right),
\end{split}
\label{eq:P1}
\end{equation}
we have the connection
\begin{align}
\Delta & =\widetilde{c}\left(I-P\right),\mbox{ and}\label{eq:P2}\\
P & =I-\frac{1}{\widetilde{c}}\Delta\label{eq:P3}
\end{align}
from (\ref{eq:F3-2}) in Corollary \ref{cor:P}. 
\begin{thm}
\label{thm:P}Let $\mathscr{H}_{E}$ (depending on $c$) be the energy
Hilbert space, and set $l^{2}\left(\widetilde{c}\right)=l^{2}\left(V,\widetilde{c}\right)=$
all functions on $V$ with inner product
\begin{equation}
\left\langle u_{1},u_{2}\right\rangle _{l^{2}\left(\widetilde{c}\right)}=\sum_{x\in V}\widetilde{c}\left(x\right)\overline{u_{1}\left(x\right)}u_{2}\left(x\right).\label{eq:P4}
\end{equation}
Then
\begin{enumerate}
\item \label{enu:p1}$\Delta$ is Hermitian in $\mathscr{H}_{E}$, but \uline{not}
in $l^{2}\left(\widetilde{c}\right)$.
\item \label{enu:p2}$P$ is Hermitian in $l^{2}\left(\widetilde{c}\right)$
and in $\mathscr{H}_{E}$.
\end{enumerate}
\end{thm}
\begin{proof}
The first half of conclusion (\ref{enu:p1}) is contained in Lemma
\ref{lem:Delta}. To show that $P$ is also Hermitian in $\mathscr{H}_{E}$,
use (\ref{eq:P3}) and the following lemma applied to $f=\frac{1}{\widetilde{c}}$. \end{proof}
\begin{lem}
Let $\Delta$, $P$, and $\mathscr{H}_{E}$ be as above, and let $f$
be a function on $V$, then 
\begin{equation}
\left\langle \left(f\Delta\right)v_{x},v_{y}\right\rangle _{\mathscr{H}_{E}}=\left\langle v_{x},\left(f\Delta\right)v_{y}\right\rangle _{\mathscr{H}_{E}}.\label{eq:P5}
\end{equation}
\end{lem}
\begin{proof}
We compute as follows, using (\ref{eq:di4}):
\begin{eqnarray*}
\mbox{LHS}_{\left(\ref{eq:P5}\right)} & = & \left\langle f\left(\cdot\right)\left(\delta_{x}-\delta_{o}\right),v_{y}\right\rangle _{\mathscr{H}_{E}}\\
 & = & \left(f\left(\cdot\right)\left(\delta_{x}-\delta_{o}\right)\right)\left(y\right)-\left(f\left(\cdot\right)\left(\delta_{x}-\delta_{o}\right)\right)\left(o\right)\\
 & = & f\left(y\right)\delta_{xy}+f\left(o\right)
\end{eqnarray*}
for all vertices $x,y\in V'=V\backslash\left\{ o\right\} $ where
$o$ is a fixed choice of base-point in the vertex set $V$. The desired
conclusion (\ref{eq:P5}) follows.
\end{proof}

\begin{proof}[Proof of Theorem \ref{thm:P}, part (\ref{enu:p2}) ]
We show that 
\[
\left\langle \left(Pu_{1}\right),u_{2}\right\rangle _{l^{2}\left(\widetilde{c}\right)}=\left\langle u_{1},\left(Pu_{2}\right)\right\rangle _{l^{2}\left(\widetilde{c}\right)}
\]
for all finitely supported functions $u_{1},u_{2}$ on $V$. 

But this follows from the assumption $\widetilde{c}\left(x\right)p_{xy}=\widetilde{c}\left(y\right)p_{yx}$
(reversible). Since 
\[
\sum_{x}\widetilde{c}\left(x\right)p_{xy}=\sum_{x}\widetilde{c}\left(y\right)p_{yx}=\widetilde{c}\left(y\right)\sum_{x}p_{yx}=\widetilde{c}\left(y\right)
\]
i.e., the function $\widetilde{c}$ is left-invariant for $P$, viz:
$\widetilde{c}P=\widetilde{c}$ viewing $P=\left(p_{xy}\right)$ as
a Markov matrix.

The final assertion, that $\Delta$ is not Hermitian in $l^{2}\left(\widetilde{c}\right)$
follows from Lemma \ref{lem:Delta}, and the fact that $\Delta$ does
not commute with the multiplication operator $f=\frac{1}{\widetilde{c}}$. \end{proof}
\begin{cor}
Let $\left(Pu\right)\left(x\right)=\sum_{y\sim x}p_{xy}u\left(y\right)$
be the transition operator, where 
\begin{equation}
p_{xy}=\frac{c_{xy}}{\widetilde{c}\left(x\right)}\label{eq:P3-1}
\end{equation}
is defined for $\left(xy\right)\in E$, $c$ is a fixed conductance
on $\left(V,E\right)$, and 
\begin{equation}
\widetilde{c}\left(x\right)=\sum_{y\sim x}c_{xy}.\label{eq:P3-2}
\end{equation}
Then $P$ is selfadjoint and contractive in $l^{2}\left(V,\widetilde{c}\right)$.\end{cor}
\begin{proof}
We proved in Lemma \ref{lem:Delta} (\ref{enu:D3}) that $\left\langle u,\Delta u\right\rangle _{l^{2}}\geq0$
holds for all $u\in l^{2}$ where $\left\langle \cdot,\cdot\right\rangle _{l^{2}}$
referes to the un-weighted $l^{2}$-inner product. The connection
between the two inner products is as follows: $\left\langle u,\widetilde{c}u\right\rangle _{l^{2}}=\left\Vert u\right\Vert _{l^{2}\left(\widetilde{c}\right)}^{2}$
which yields the following:

Using (\ref{eq:P3-1}) and (\ref{eq:P3-2}), we get 
\begin{equation}
Pu=u-\frac{1}{\widetilde{c}}\Delta u,\label{eq:P3-3}
\end{equation}
so $\widetilde{c}Pu=\widetilde{c}u-\Delta u$, and as a point-wise
identity on $V$. Hence 
\begin{align*}
\left\langle u,Pu\right\rangle _{l^{2}\left(\widetilde{c}\right)} & =\left\langle u,\widetilde{c}u-\Delta u\right\rangle _{l^{2}}\\
 & =\left\Vert u\right\Vert _{l^{2}\left(\widetilde{c}\right)}^{2}-\left\langle u,\Delta u\right\rangle _{l^{2}}\\
 & \leq\left\Vert u\right\Vert _{l^{2}\left(\widetilde{c}\right)}^{2},\;\left(\mbox{by Lemma \ref{lem:Delta} \ensuremath{\left(\ref{enu:D3}\right)}}\right)
\end{align*}
holds for all $u\in l^{2}\left(\widetilde{c}\right)$.

Since we also proved that $P$ (see eq. (\ref{eq:P3-3})) is $l^{2}\left(\widetilde{c}\right)$-Hermitian,
we conclude that it is contractive and selfadjoint in the Hilbert
space $l^{2}\left(\widetilde{c}\right)$, as claimed.\end{proof}
\begin{lem}
\label{lem:hf}Let $\left(V,E,c,p\right)$, $\Delta$, and $P$ be
as above. Then a function $u$ on $V$ satisfies $\Delta u=0$ if
and only if $Pu=u$. \end{lem}
\begin{proof}
Immediate from $\Delta u=\tilde{c}\left(u-Pu\right)$ (eq. (\ref{eq:P2})).\end{proof}
\begin{cor}
Let $\left(V,E,c\right)$ be as above, and set 
\begin{equation}
p_{xy}:=\dfrac{c_{xy}}{\widetilde{c}\left(x\right)},\;\left(xy\right)\in E;\label{eq:P2-1}
\end{equation}
where $\widetilde{c}\left(x\right):=\sum_{y\sim x}c_{xy}$. Fix $o\in V$,
and consider the dipoles $\left(v_{x}\right)_{x\in V'}$, $V'=V\backslash\left\{ o\right\} $,
where 
\begin{equation}
\left\langle v_{x},u\right\rangle _{\mathscr{H}_{E}}=u\left(x\right)-u\left(o\right),\;\forall x\in V'.\label{eq:P1-1}
\end{equation}
Setting 
\begin{equation}
\left(Pu\right)\left(x\right)=\sum_{y\sim x}p_{xy}u\left(y\right),\label{eq:P2-2}
\end{equation}
we get 
\begin{equation}
Pv_{x}=\sum_{y\sim x}p_{xy}v_{y};\:\mbox{ and}\label{eq:P1-2}
\end{equation}
the following implication holds for the dense subspace $\mathscr{D}_{E}$
in $\mathscr{H}_{E}$:
\begin{equation}
u\in\mathscr{D}_{E}\Longrightarrow Pu\in\mathscr{D}_{E};\;\mbox{where}\label{eq:P1-4}
\end{equation}
\begin{equation}
\mathscr{D}_{E}:=\left\{ \sum_{x}\xi_{x}v_{x}\:\big|\:\xi_{x}\in\mathbb{C},\;\mbox{finitely supported on \ensuremath{V',}s.t. }\sum_{x}\xi_{x}=0\right\} \label{eq:P1-3}
\end{equation}
\end{cor}
\begin{proof}
A direct computation shows that (\ref{eq:P1-2}) must hold as an identity
on functions on $V$, up to an additive constant. 

Our assertion is that working in the Hilbert space $\mathscr{H}_{E}$
implies that the additive constant is zero. This amounts to verification
of the implication (\ref{eq:P1-4}); i.e., that 
\begin{equation}
\sum_{x}\xi_{x}=0\Longrightarrow\sum_{y}\left(\sum_{x}\xi_{x}p_{xy}\right)=0\label{eq:P1-5}
\end{equation}
for all finitely supported functions. But we have 
\[
\sum_{y}\left(\sum_{x}\xi_{x}p_{xy}\right)=\sum_{x}\xi_{x}\left(\sum_{y}p_{xy}\right)=\sum_{x}\xi_{x}=0
\]
which is the desired assertion (\ref{eq:P1-5}). Hence (\ref{eq:P1-4})
follows.
\end{proof}

\section{\label{sec:eg}Examples}

The purpose of the first example is multi-fold:

First, by picking an infinite arithmetic progression of points on
the line as vertex set $V$, and nearest neighbors, an assignment
of conductance simply amounts to a function on the edges $(n,n+1)$,
and we get non-trivial models where explicit formulas are possible
and transparent. For example, we can write down the dipoles $v_{xy}$
as functions on $V$, and the corresponding resistance metric; see
the formulas relating to Figure \ref{fig:1dipole}. Among all the
conductance functions we characterize the cases of reversible Markov
models where the left/right transition probabilities are the same
for all vertex points. In section \ref{ex:btree} (the binomial model)
we accomplish the same characterization of the cases of reversible
Markov models where the left/right transition probabilities are the
same for all vertex points, but now every vertex in the binary tree
has three nearest neighbors.

In section \ref{ex:tri} we give a finite graph $(V,E,c)$ as a triangular
configuration, conductance $c$ defined on the edges of the triangle,
and we find the Parseval frame (thereby illustrating Theorem \ref{thm:eframe}).

The examples below illustrate the following: When a graph network
$(V,E,c)$ is infinite, then the dipoles $v_{xy}$ as functions on
$V$ will not lie in the Hilbert space $l^{2}(V)$. Hence another
justification for the energy Hilbert space $\mathscr{H}_{E}$.

\subsection{\label{ex:1d}$V=\left\{ 0\right\} \cup\mathbb{Z}_{+}$}

Consider $G=\left(V,E,c\right)$, where $V=\left\{ 0\right\} \cup\mathbb{Z}_{+}$.
Observation: Every sequence $a_{1},a_{2},\ldots$ in $\mathbb{R}_{+}$
defines a conductance $c_{n-1,n}:=a_{n}$, $n\in\mathbb{Z}_{+}$,
i.e., 
\[
\xymatrix{0\ar@{<->}[r]_{a_{1}} & 1\ar@{<->}[r]_{a_{2}} & 2\ar@{<->}[r]_{a_{3}} & 3 & \cdots & n\ar@{<->}[r]_{a_{n+1}} & n+1}
\cdots
\]

The dipole vectors $v_{xy}$ (for $x,y\in\mathbb{N}$) are given by
\[
v_{xy}\left(z\right)=\begin{cases}
0 & \mbox{if \ensuremath{z\leq x}}\\
-\sum_{k=x+1}^{z}\frac{1}{a_{k}} & \mbox{if \ensuremath{x<z<y}}\\
-\sum_{k=x+1}^{y}\frac{1}{a_{k}} & \mbox{if \ensuremath{z\geq y}}
\end{cases}
\]
See Fig \ref{fig:1dipole}. 

\begin{figure}[H]
\includegraphics[scale=0.5]{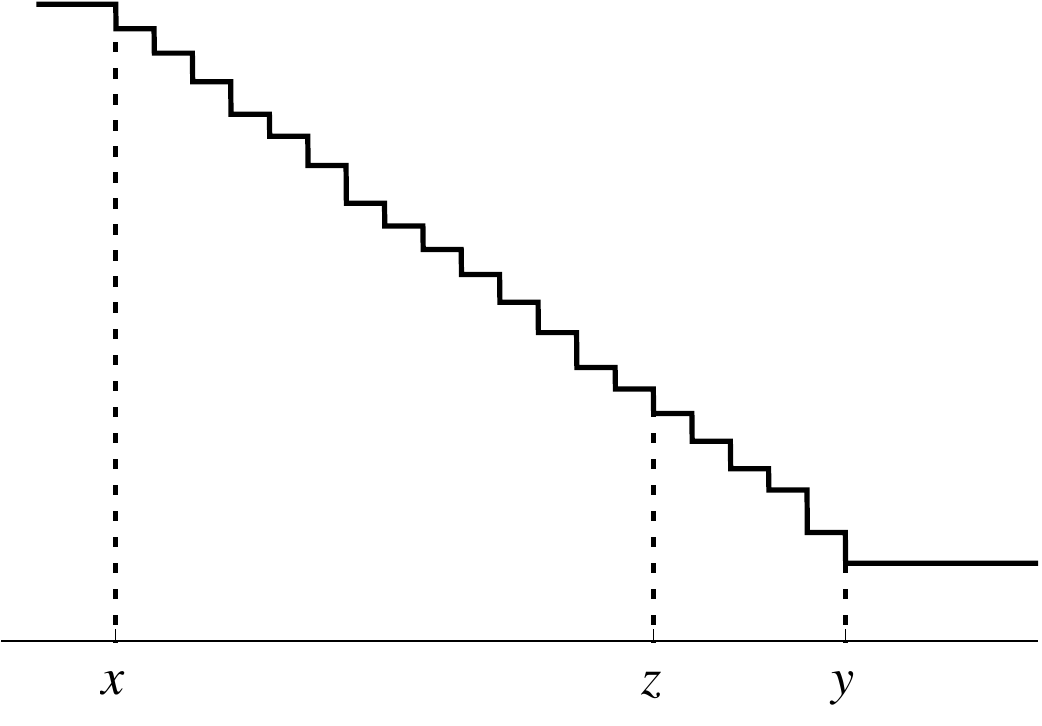}

\protect\caption{\label{fig:1dipole}The dipole $v_{xy}$.}
\end{figure}

It follows from Lemma \ref{lem:metric} that the resistance metric
$dist\left(=d_{c}=d_{a}\right)$ is as follows
\[
dist\left(x,y\right)=\begin{cases}
0 & \mbox{if \ensuremath{x=y}}\\
\underset{\sum_{x<k\leq y}\frac{1}{a_{k}}}{\underbrace{\frac{1}{a_{x+1}}+\cdots+\frac{1}{a_{y}}}} & \mbox{if \ensuremath{x<y}}
\end{cases}
\]

Note that $V=\mathbb{Z}_{+}\cup\left\{ 0\right\} $ with the resistance
metric described above yields a \uline{bounded} metric space if
and only if 
\begin{equation}
\sum_{n=1}^{\infty}\frac{1}{a_{n}}<\infty.\label{eq:ex1-1}
\end{equation}

The corresponding graph Laplacian has the following matrix representation:

\begin{equation}
\begin{bmatrix}a_{1} & -a_{1}\\
-a_{1} & a_{1}+a_{2} & -a_{2}\\
 & -a_{2} & a_{2}+a_{3} & -a_{3} &  &  & \Huge\mbox{0}\\
 &  & -a_{3} & a_{3}+a_{4} & \ddots\\
 &  &  & \ddots & \ddots & -a_{n}\\
 &  &  &  & -a_{n} & a_{n}+a_{n+1} & -a_{n+1}\\
 & \Huge\mbox{0} &  &  &  & -a_{n+1} & \ddots & \ddots\\
 &  &  &  &  &  & \ddots & \ddots\\
 &  &  &  &  &  & \ddots & \ddots
\end{bmatrix}\label{eq:Lma1}
\end{equation}
That is, 
\begin{equation}
\begin{cases}
\left(\Delta u\right)_{0} & =a_{1}\left(u_{0}-u_{1}\right)\\
\left(\Delta u\right)_{n} & =a_{n}\left(u_{n}-u_{n-1}\right)+a_{n+1}\left(u_{n}-u_{n+1}\right)\\
 & =\left(a_{n}+a_{n+1}\right)u_{n}-a_{n}u_{n-1}-a_{n+1}u_{n+1},\;\forall n\in\mathbb{Z}_{+}.
\end{cases}\label{eq:Le1}
\end{equation}

\begin{lem}
Let $G=\left(V,c,E\right)$ be as above, where $a_{n}:=c_{n-1,n}$,
$n\in\mathbb{Z}_{+}$. Then $u\in\mathscr{H}_{E}$ is the solution
to $\Delta u=-u$ (i.e., $u$ is a defect vector of $\Delta$) if
and only if $u$ satisfies the following equation:
\begin{equation}
\sum_{n=1}^{\infty}a_{n}\left\langle v_{n-1,n},u\right\rangle _{\mathscr{H}_{E}}\left(\delta_{n-1}\left(s\right)-\delta_{n}\left(s\right)+v_{n-1,n}\left(s\right)\right)=0,\;\forall s\in\mathbb{Z}_{+};\label{eq:1ddef}
\end{equation}
where 
\begin{equation}
\left\Vert u\right\Vert _{\mathscr{H}_{E}}^{2}=\sum_{n=1}^{\infty}a_{n}\left|\left\langle v_{n-1,n},u\right\rangle _{\mathscr{H}_{E}}\right|^{2}<\infty.\label{eq:1ddef-1}
\end{equation}
\end{lem}
\begin{proof}
By Theorem \ref{thm:eframe}, the set $\left\{ \sqrt{a_{n}}v_{n-1,n}\right\} _{n=1}^{\infty}$
forms a Parseval frame in $\mathscr{H}_{E}$. In fact, the dipole
vectors are 
\begin{equation}
v_{n-1,n}\left(s\right)=\begin{cases}
0 & s\leq n-1\\
-\frac{1}{a_{n}} & s\geq n
\end{cases};n=1,2,\ldots\label{eq:2}
\end{equation}
and so $\left\{ \sqrt{a_{n}}v_{n-1,n}\right\} _{n=1}^{\infty}$ forms
an ONB in $\mathscr{H}_{E}$; and $u\in\mathscr{H}_{E}$ has the representation
\[
u=\sum_{n=1}^{\infty}a_{n}\left\langle v_{n-1,n},u\right\rangle _{\mathscr{H}_{E}}v_{n-1,n}
\]
see (\ref{eq:en4}). Therefore, $\Delta u=-u$ if and only if 
\begin{align*}
\sum_{n=1}^{\infty}a_{n}\left\langle v_{n-1,n},u\right\rangle _{\mathscr{H}_{E}}\left(\delta_{n-1}\left(s\right)-\delta_{n}\left(s\right)\right) & =-\sum_{n=1}^{\infty}a_{n}\left\langle v_{n-1,n},u\right\rangle _{\mathscr{H}_{E}}v_{n-1,n}\left(s\right)
\end{align*}
for all $s\in\mathbb{Z}_{+}$, which is the assertion.\end{proof}
\begin{conjecture}
Consider $\Delta$ as above as an operator in $\mathscr{H}_{E}$ (depending
on $c_{n,n-1}=a_{n}$). Then $\Delta$ is essentially selfadjoint
(in $\mathscr{H}_{E}$) if and only if $\sum_{n=1}^{\infty}\frac{1}{a_{n}}=\infty$.
If (\ref{eq:ex1-1}) holds, the indices are $\left(1,1\right)$. \end{conjecture}
\begin{rem}
Below we compute the deficiency space in an example with index values
$\left(1,1\right)$.\end{rem}
\begin{lem}
\label{lem:Q11}Let $\left(V,E,c=\left\{ a_{n}\right\} \right)$ be
as above. Let $Q>1$ and set $a_{n}:=Q^{n}$, $n\in\mathbb{Z}_{+}$;
then $\Delta$ has deficiency indices $\left(1,1\right)$.\end{lem}
\begin{proof}
Suppose $\Delta u=-u$, $u\in\mathscr{H}_{E}$. Then, 
\begin{align*}
-u_{1} & =Q\left(u_{1}-u_{0}\right)+Q^{2}\left(u_{1}-u_{2}\right)\Longleftrightarrow u_{2}=\left(\frac{1}{Q^{2}}+\frac{1+Q}{Q}\right)u_{1}-\frac{1}{Q}u_{0}\\
-u_{2} & =Q^{2}\left(u_{2}-u_{1}\right)+Q^{3}\left(u_{2}-u_{3}\right)\Longleftrightarrow u_{3}=\left(\frac{1}{Q^{3}}+\frac{1+Q}{Q}\right)u_{2}-\frac{1}{Q}u_{1}
\end{align*}
and by induction, 
\[
u_{n+1}=\left(\frac{1}{Q^{n+1}}+\frac{1+Q}{Q}\right)u_{n}-\frac{1}{Q}u_{n-1},\; n\in\mathbb{Z}_{+}
\]
i.e., $u$ is determined by the following matrix equation:

\[
\begin{bmatrix}u_{n+1}\\
u_{n}
\end{bmatrix}=\begin{bmatrix}\frac{1}{Q^{n+1}}+\frac{1+Q}{Q} & -\frac{1}{Q}\\
1 & 0
\end{bmatrix}\begin{bmatrix}u_{n}\\
u_{n-1}
\end{bmatrix}
\]

The eigenvalues of the coefficient matrix are 
\begin{align*}
\lambda_{\pm} & =\frac{1}{2}\left(\frac{1}{Q^{n+1}}+\frac{1+Q}{Q}\pm\sqrt{\left(\frac{1}{Q^{n+1}}+\frac{1+Q}{Q}\right)^{2}-\frac{4}{Q}}\right)\\
 & \sim\frac{1}{2}\left(\frac{1+Q}{Q}\pm\left(\frac{Q-1}{Q}\right)\right)=\begin{cases}
1\\
\dfrac{1}{Q}
\end{cases}\mbox{as \ensuremath{n\rightarrow\infty}.}
\end{align*}
Equivalently, as $n\rightarrow\infty$, we have
\[
u_{n+1}\sim\left(\frac{1+Q}{Q}\right)u_{n}-\frac{1}{Q}u_{n-1}=\left(1+\frac{1}{Q}\right)u_{n}-\frac{1}{Q}u_{n-1}
\]
and so 
\[
u_{n+1}-u_{n}\sim\frac{1}{Q}\left(u_{n}-u_{n-1}\right).
\]
Therefore, for the tail-summation, we have:
\[
\sum_{n}Q^{n}\left(u_{n+1}-u_{n}\right)^{2}=\mbox{const}\sum_{n}\frac{\left(Q-1\right)^{2}}{Q^{n+2}}<\infty
\]
which implies $\left\Vert u\right\Vert _{\mathscr{H}_{E}}<\infty$.
\textcolor{blue}{}
\end{proof}
Next, we give a random walk interpretation of Lemma \ref{lem:Q11}.
See Remark \ref{rem:rw}, and Fig \ref{fig:tp}.
\begin{rem}[Harmonic functions in $\mathscr{H}_{E}$]
 \label{rem:Qharm}Note that in Example \ref{ex:1d} (Lemma \ref{lem:Q11}),
the space of harmonic functions in $\mathscr{H}_{E}$ is one-dimensional;
in fact if $Q>1$ is fixed, then 
\[
\left\{ u\in\mathscr{H}_{E}\:\big|\:\Delta u=0\right\} 
\]
is spanned by $u=\left(u_{n}\right)_{n=0}^{\infty}$, $u_{n}=\frac{1}{Q^{n}}$,
$n\in\mathbb{N}$; and of course $\Vert1/Q^{n}\Vert_{\mathscr{H}_{E}}^{2}<\infty$.\end{rem}
\begin{proof}
This is immediate from Lemma \ref{lem:hf}.\end{proof}
\begin{rem}
\label{rem:domFri}For the domain of the Friedrichs extension $\Delta_{Fri}$,
we have:
\begin{equation}
dom(\Delta_{Fri})=\left\{ f\in\mathscr{H}_{E}\:|\:\left(f\left(x\right)-f\left(x+1\right)\right)Q^{x}\in l^{2}\left(\mathbb{Z}_{+}\right)\right\} \label{eq:1d-2-1}
\end{equation}
i.e., 
\[
dom(\Delta_{Fri})=\left\{ f\in\mathscr{H}_{E}\:|\:\sum_{x=0}^{\infty}\left|f\left(x\right)-f\left(x+1\right)\right|^{2}Q^{2x}<\infty\right\} .
\]
\end{rem}
\begin{proof}
By Theorem \ref{thm:eframe}, we have the following representation,
valid for all $f\in\mathscr{H}_{E}$: 
\begin{align*}
f & =\sum_{x}\left\langle f,Q^{\frac{x}{2}}v_{\left(x,x+1\right)}\right\rangle _{\mathscr{H}_{E}}Q^{\frac{x}{2}}v_{\left(x,x+1\right)}\\
 & =\sum_{x}\left(f\left(x\right)-f\left(x+1\right)\right)Q^{x}v_{\left(x,x+1\right)};
\end{align*}
and
\[
\left\langle f,\Delta f\right\rangle _{\mathscr{H}_{E}}=\sum_{x}\left|f\left(x\right)-f\left(x+1\right)\right|^{2}Q^{2x}.
\]
The desired conclusion (\ref{eq:1d-2-1}) now follows from Theorem
\ref{thm:Deltaf} above, and the characterization of $\Delta_{Fri}$,
see e.g. \cite{DS88b,AG93}.\end{proof}
\begin{defn}
Let $G=\left(V,E,c\right)$ be a connected graph. The set of transition
probabilities $\left(p_{xy}\right)$ is said to be reversible if there
exists $c:V\rightarrow\mathbb{R}_{+}$ s.t. 
\begin{equation}
c\left(x\right)p_{xy}=c\left(y\right)p_{yx};\label{eq:tp1}
\end{equation}
and then 
\begin{equation}
c_{xy}:=c\left(x\right)p_{xy}\label{eq:tpc}
\end{equation}
is a system of conductance. Conversely, for a system of conductance
$\left(c_{xy}\right)$ we set 
\begin{align}
c\left(x\right) & :=\sum_{y\sim x}c_{xy},\;\mbox{and}\label{eq:cond1}\\
p_{xy} & :=\frac{c_{xy}}{c\left(x\right)}\label{eq:tp-1}
\end{align}
and so $\left(p_{xy}\right)$ is a set of transition probabilities.
See Fig \ref{fig:tp2} below.

\begin{figure}[H]
\begin{tabular}{cc}
\includegraphics[scale=0.45]{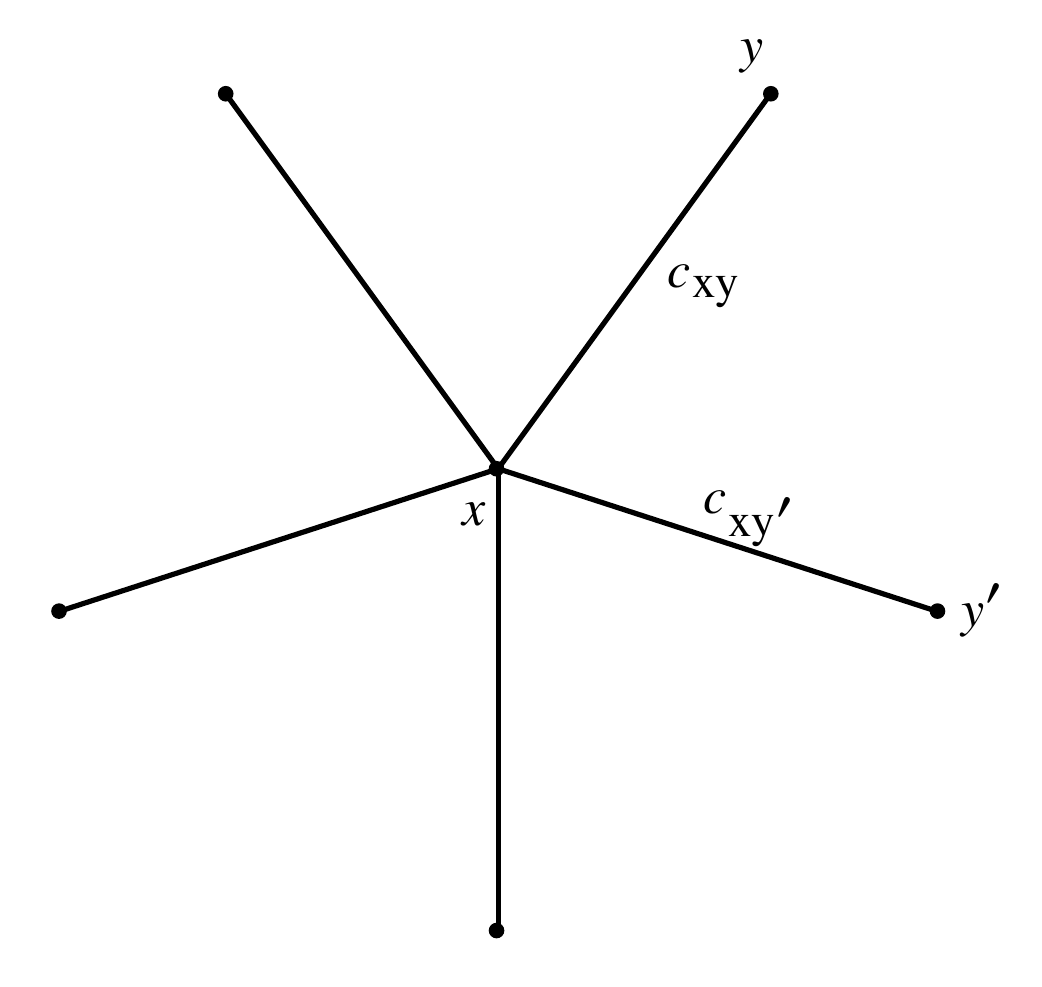} & \includegraphics[scale=0.45]{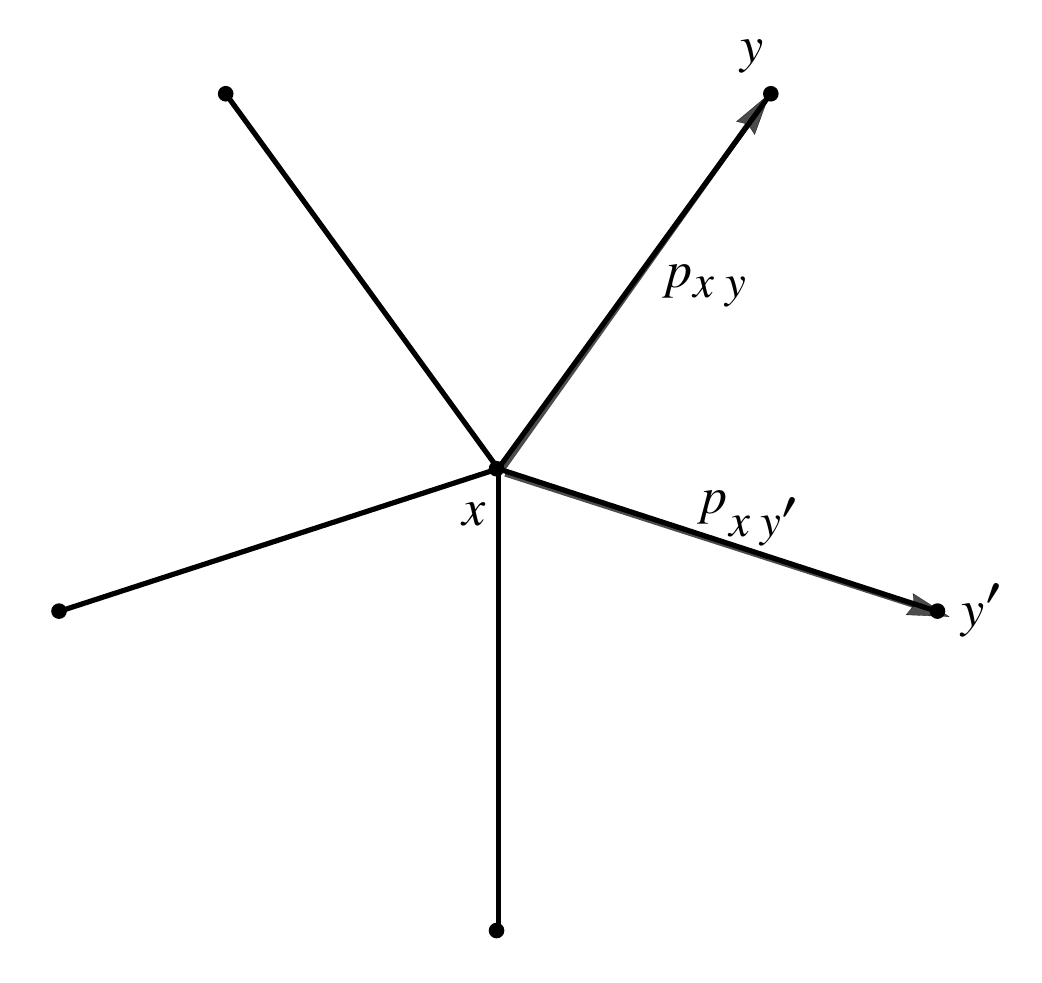}\tabularnewline
$c_{xy}$, $y\sim x$ & transition probabilities\tabularnewline
\end{tabular}

\protect\caption{\label{fig:tp2}neighbors of $x$}

\end{figure}

\end{defn}
Recall the graph Laplacian in (\ref{eq:Le1}) can be written as 
\begin{equation}
\left(\Delta u\right)_{n}=c\left(n\right)\left(u_{n}-p_{-}\left(n\right)u_{n-1}-p_{+}\left(n\right)u_{n+1}\right),\;\forall n\in\mathbb{Z}_{+};\label{eq:Le2}
\end{equation}
where 
\begin{equation}
c\left(n\right):=a_{n}+a_{n+1}\label{eq:Lec}
\end{equation}
and 
\begin{equation}
p_{-}\left(n\right):=\frac{a_{n}}{c\left(n\right)},\; p_{+}\left(n\right):=\frac{a_{n+1}}{c\left(n\right)}\label{eq:Lep}
\end{equation}
are the left/right transition probabilities, as shown in Fig \ref{fig:tp1}.

\begin{figure}[H]
\includegraphics[scale=0.6]{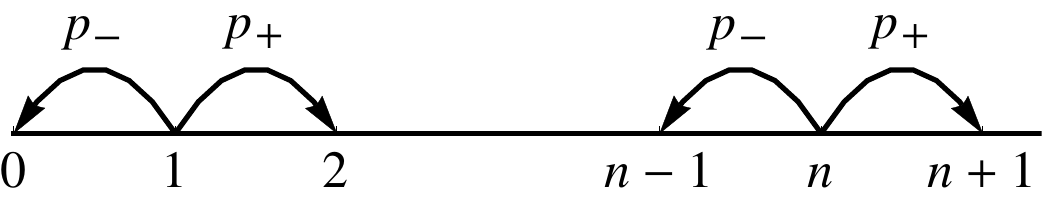}

\protect\caption{\label{fig:tp1}The transition probabilities $p_{+},p_{-}$, in the
case of constant transition probabilities, i.e., $p_{+}\left(n\right)=p_{+}$,
and $p_{-}\left(n\right)=p_{-}$ for all $n\in\mathbb{Z}_{+}$.}

\end{figure}

In the case $a_{n}=Q^{n}$, $Q>1$, as in Lemma \ref{lem:Q11}, we
have 
\begin{equation}
c\left(n\right):=Q^{n}+Q^{n+1},\;\mbox{and }\label{eq:cn}
\end{equation}
\begin{align}
p_{+} & :=p_{+}\left(n\right)=\frac{Q^{n+1}}{Q^{n}+Q^{n+1}}=\frac{Q}{1+Q}\label{eq:pplus}\\
p_{-} & :=p_{-}\left(n\right)=\frac{Q^{n}}{Q^{n}+Q^{n+1}}=\frac{1}{1+Q}\label{eq:pmimus}
\end{align}
For all $n\in\mathbb{Z}_{+}\cup\left\{ 0\right\} $, set 
\begin{equation}
\left(Pu\right)_{n}:=p_{-}u_{n-1}+p_{+}u_{n+1}.\label{eq:cp}
\end{equation}
Note $\left(Pu\right)_{0}=u_{1}$. By (\ref{eq:Le2}), we have 
\begin{equation}
\Delta=c\left(1-P\right).\label{eq:dp}
\end{equation}

In particular, $p_{+}>\frac{1}{2}$, i.e., a random walker has probability
$>\frac{1}{2}$ of moving to the right. It follows that 
\[
\underset{=\mbox{ dist to \ensuremath{\infty}}}{\underbrace{\mbox{travel time}\left(n,\infty\right)}}<\infty;
\]
and so $\Delta$ is not essentially selfadjoint, i.e., indices $\left(1,1\right)$.
\begin{lem}
Let $\left(V,E,\Delta(=\Delta_{c})\right)$ be as above, where the
conductance $c$ is given by $c_{n-1,n}=Q^{n}$, $n\in\mathbb{Z}_{+}$,
$Q>1$ (see Lemma \ref{lem:Q11}). For all $\lambda>0$, there exists
$f_{\lambda}\in\mathscr{H}_{E}$ satisfying $\Delta f_{\lambda}=\lambda f_{\lambda}$.\end{lem}
\begin{proof}
By (\ref{eq:dp}), we have $\Delta f_{\lambda}=\lambda f_{\lambda}\Longleftrightarrow Pf_{\lambda}=\left(1-\frac{\lambda}{c}\right)f_{\lambda}$,
i.e., 
\[
\frac{1}{1+Q}f_{\lambda}\left(n-1\right)+\frac{Q}{1+Q}f_{\lambda}\left(n+1\right)=\left(1-\frac{\lambda}{Q^{n-1}\left(1+Q\right)}\right)f_{\lambda}\left(n\right)
\]
and so 
\begin{equation}
f_{\lambda}\left(n+1\right)=\left(\frac{1+Q}{Q}-\frac{\lambda}{Q^{n}}\right)f_{\lambda}\left(n\right)-\frac{1}{Q}f_{\lambda}\left(n-1\right).\label{eq:Fe0}
\end{equation}
This corresponds to the following matrix equation:

\begin{align*}
\begin{bmatrix}f\left(n+1\right)\\
f\left(n\right)
\end{bmatrix} & =\begin{bmatrix}\frac{1+Q}{Q}-\frac{\lambda}{Q^{n}} & -\frac{1}{Q}\\
1 & 0
\end{bmatrix}\begin{bmatrix}f\left(n\right)\\
f\left(n-1\right)
\end{bmatrix}\\
 & \sim\begin{bmatrix}\frac{1+Q}{Q} & -\frac{1}{Q}\\
1 & 0
\end{bmatrix}\begin{bmatrix}f\left(n\right)\\
f\left(n-1\right)
\end{bmatrix},\;\mbox{as \ensuremath{n\rightarrow\infty}.}
\end{align*}
The eigenvalues of the coefficient matrix are given by 
\[
\lambda_{\pm}\sim\frac{1}{2}\left(\frac{1+Q}{Q}\pm\left(\frac{Q-1}{Q}\right)\right)=\begin{cases}
1\\
\dfrac{1}{Q}
\end{cases}\mbox{as \ensuremath{n\rightarrow\infty}.}
\]
That is, as $n\rightarrow\infty$, 
\[
f_{\lambda}\left(n+1\right)\sim\left(\frac{1+Q}{Q}\right)f_{\lambda}\left(n\right)-\frac{1}{Q}f_{\lambda}\left(n-1\right);
\]
i.e.,
\begin{equation}
f_{\lambda}\left(n+1\right)\sim\frac{1}{Q}f_{\lambda}\left(n\right);\label{eq:Fe1}
\end{equation}
and so the tail summation of $\left\Vert f_{\lambda}\right\Vert _{\mathscr{H}_{E}}^{2}$
is finite. (See the proof of Lemma \ref{lem:Q11}.) We conclude that
$f_{\lambda}\in\mathscr{H}_{E}$.\end{proof}
\begin{cor}
Let $\left(V,E,\Delta\right)$ be as in the lemma. The Friedrichs
extension $\Delta_{Fri}$ has continuous spectrum $[0,\infty)$.\end{cor}
\begin{proof}
Fix $\lambda\geq0$. We prove that if $\Delta f_{\lambda}=\lambda f_{\lambda}$,
$f\in\mathscr{H}_{E}$, then $f_{\lambda}\notin dom(\Delta_{Fri})$. 

Note for $\lambda=0$, $f_{0}$ is harmonic, and so $f_{0}=k\left(\frac{1}{Q^{n}}\right)_{n=0}^{\infty}$
for some constant $k\neq0$. See Remark \ref{rem:Qharm}. It follows
from (\ref{eq:1d-2-1}) that $f_{0}\notin dom(\Delta_{Fri})$. 

The argument for $\lambda>0$ is similar. Since as $n\rightarrow\infty$,
$f_{\lambda}\left(n\right)\sim\frac{1}{Q^{n}}$ (eq. (\ref{eq:Fe1})),
so by (\ref{eq:1d-2-1}) again, $f_{\lambda}\notin dom(\Delta_{Fri})$.

However, if $\lambda_{0}<\lambda_{1}$ in $[0,\infty)$ then 
\begin{equation}
\int_{\lambda_{0}}^{\lambda_{1}}f_{\lambda}\left(\cdot\right)d\lambda\in dom(\Delta_{Fri})\label{eq:Fe2}
\end{equation}
and so every $f_{\lambda}$, $\lambda\in[0,\infty)$, is a generalized
eigenfunction, i.e., the spectrum of $\Delta_{Fri}$ is purely continuous
with Lebesgue measure, and multiplicity one. 

The verification of (\ref{eq:Fe2}) follows from eq. (\ref{eq:Fe0}),
i.e., 
\begin{equation}
f_{\lambda}\left(n+1\right)=\left(\frac{1+Q}{Q}-\frac{\lambda}{Q^{n}}\right)f_{\lambda}\left(n\right)-\frac{1}{Q}f_{\lambda}\left(n-1\right).\label{eq:Fe3}
\end{equation}
Set 
\begin{equation}
F_{\left[\lambda_{0},\lambda_{1}\right]}:=\int_{\lambda_{0}}^{\lambda_{1}}f_{\lambda}\left(\cdot\right)d\lambda.\label{eq:Fe4}
\end{equation}
Then by (\ref{eq:Fe3}) and (\ref{eq:Fe4}),
\[
F_{\left[\lambda_{0},\lambda_{1}\right]}\left(n+1\right)=\frac{1+Q}{Q}F_{\left[\lambda_{0},\lambda_{1}\right]}\left(n\right)-\frac{1}{Q^{n}}\int_{\lambda_{0}}^{\lambda_{1}}\lambda f_{\lambda}\left(n\right)d\lambda-\frac{1}{Q}F_{\left[\lambda_{0},\lambda_{1}\right]}\left(n-1\right)
\]
and $\int_{\lambda_{0}}^{\lambda_{1}}\lambda f_{\lambda}d\lambda$
is computed using integration by parts.\end{proof}
\begin{rem}[Krein extension]
Set 
\begin{eqnarray*}
dom(\Delta^{Harm}) & := & dom(\Delta)+\mbox{Harmonic functions}\\
\Delta^{Harm} & := & \Delta^{*}\big|_{dom(\Delta^{Harm})};
\end{eqnarray*}
Then $\Delta^{Harm}\supset\Delta$ is a well-defined selfadjoint extension
of $\Delta$. It is semibounded, since 
\[
\left\langle \varphi+h,\Delta^{Harm}\left(\varphi+h\right)\right\rangle _{\mathscr{H}_{E}}=\left\langle \varphi+h,\Delta\varphi\right\rangle _{\mathscr{H}_{E}}=\left\langle \varphi,\Delta\varphi\right\rangle _{\mathscr{H}_{E}}\geq0
\]
for all $\varphi\in dom(\Delta)$, and $h$ harmonic. In fact $\Delta^{Harm}$
is the Krein extension. 
\end{rem}

\subsection{\label{ex:btree}A reversible walk on the binary tree (the binomial
model)}

Consider the binary tree as a set $V$ of vertices. To get a graph
$G=\left(V,E\right)$, take for edges $E$ the nearest neighbor lines
as follows:
\begin{equation}
V:=\left\{ o=\phi,\left(x_{1}x_{2}\cdots x_{n}\right),x_{i}\in\left\{ 0,1\right\} ,1\leq i\leq n\right\} ,\mbox{ and}\label{eq:bt1}
\end{equation}
\begin{align*}
 & E\left(\phi\right)=\left\{ 0,1\right\} \\
 & E\left(\left(x_{1}x_{2}\cdots x_{n}\right)\right)=\Big\{\underset{=x^{*}}{\underbrace{\left(x_{1}\cdots x_{n-1}\right)}},\underset{\text{extended words}}{\underbrace{\left(x0\right),\left(x1\right)}}\Big\},\mbox{ see Fig \ref{fig:bt}.}
\end{align*}

\begin{figure}[H]
\begin{tabular}{>{\centering}p{0.45\columnwidth}c}
\includegraphics[scale=0.45]{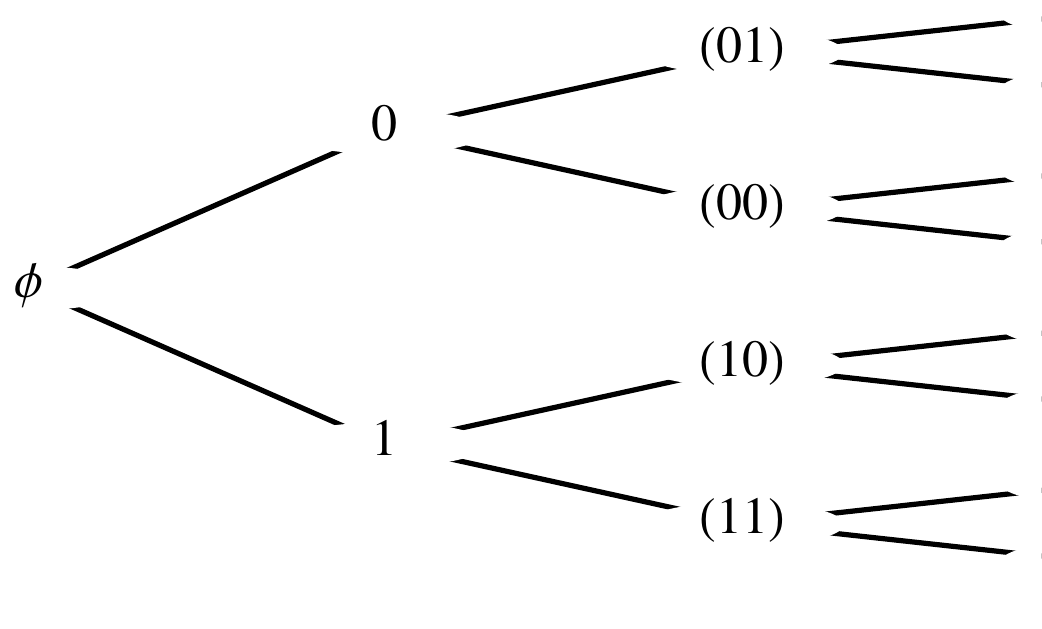} & \includegraphics[scale=0.45]{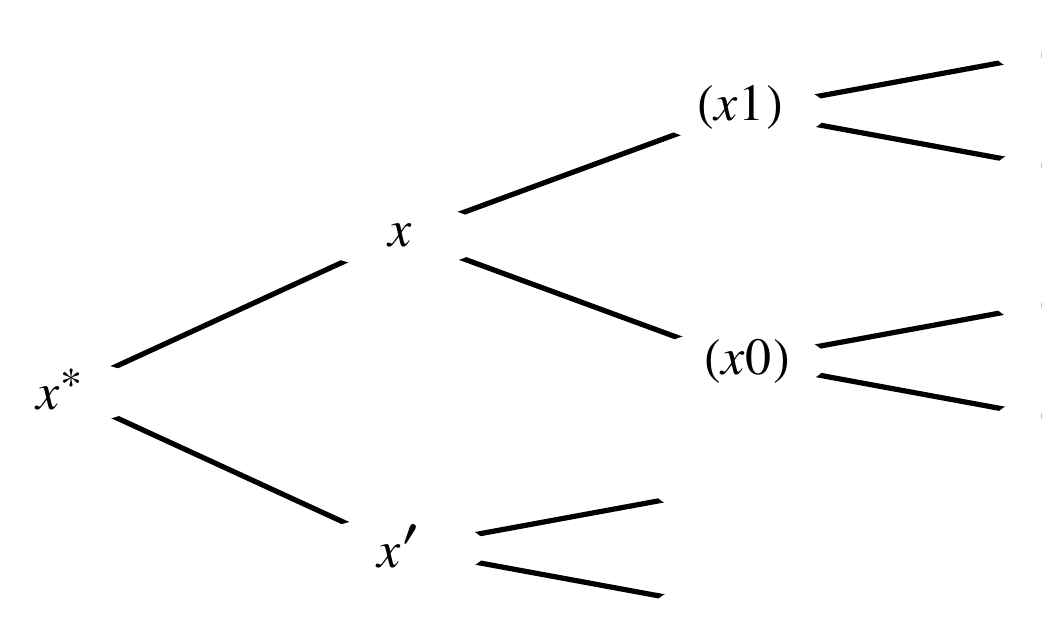}\tabularnewline
\end{tabular}

\protect\caption{\label{fig:bt}The binary tree model, three nearest neighbors.}

\end{figure}

Now fix transition probabilities at $o=\phi$, and at the vertices
$V'=V\backslash\left\{ \phi\right\} $ as follows:
\begin{equation}
\begin{cases}
\mbox{Prob}\left(x\rightarrow\left(x0\right)\right) & =p_{0}\\
\mbox{Prob}\left(x\rightarrow\left(x1\right)\right) & =p_{1}\\
\mbox{Prob}\left(x\rightarrow\left(x^{*}\right)\right) & =p_{-}
\end{cases}\label{eq:bt2}
\end{equation}
with $0<p_{0}<1$ (see Fig \ref{fig:bt}); such that $p_{0}+p_{1}+p_{-}=1$,
see Fig \ref{fig:btt}.

\begin{figure}[H]
\includegraphics[scale=0.5]{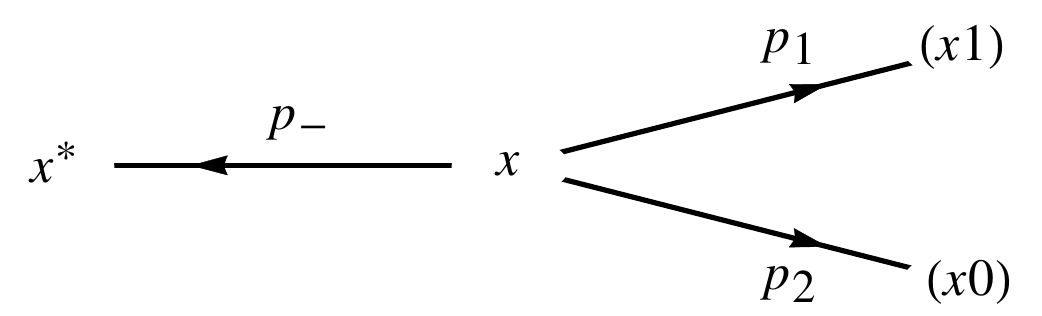}

\protect\caption{\label{fig:btt}Transition probabilities at a vertex $x\in V'$. The
reversible case with three nearest neighbors.}
\end{figure}

For $x\in V'$, set 
\begin{equation}
\begin{split}F_{0}\left(x\right)=\#\left\{ i\:|\: x_{i}=0\right\} \\
F_{1}\left(x\right)=\#\left\{ i\:|\: x_{i}=1\right\} 
\end{split}
;\label{eq:bt3}
\end{equation}
and 
\begin{equation}
\begin{split}\left|x\right| & :=x_{1}+x_{2}+\cdots+x_{n};\mbox{ so that}\\
 & F_{0}\left(x\right)+F_{1}\left(x\right)=\left|x\right|,\;\forall x\in V'.
\end{split}
\label{eq:bt4}
\end{equation}
Further, define the function $c:V\rightarrow\mathbb{R}_{+}$ as follows:
\begin{equation}
c\left(x\right)=\frac{p_{0}^{F_{0}\left(x\right)}p_{1}^{F_{1}\left(x\right)}}{p_{-}^{\left|x\right|}},\;\forall x\in V'.\label{eq:bt5}
\end{equation}

\begin{lem}
With the transition probabilities defined in (\ref{eq:bt2}), it follows
that the corresponding walk on $V$ is reversible via the function
$c:V\rightarrow\mathbb{R}_{+}$ defined in (\ref{eq:bt5}), i.e.,
we have the following identity for any edge $\left(xy\right)$ in
$G=\left(V,E\right)$:
\begin{equation}
c\left(x\right)\mbox{Prob}\left(x\rightarrow y\right)=c\left(y\right)\mbox{Prob}\left(y\rightarrow x\right).\label{eq:bt6}
\end{equation}
\end{lem}
\begin{proof}
This follows from a direct inspection of the cases; see also Figures
\ref{fig:bt} and \ref{fig:btt}.
\end{proof}
For $\left(xy\right)\in E$, using (\ref{eq:bt5})-(\ref{eq:bt6}),
set $c_{xy}:=c\left(x\right)\mbox{Prob}\left(x\rightarrow y\right)$,
and 
\begin{equation}
\left(\Delta u\right)\left(x\right):=\sum_{y\sim x}c_{xy}\left(u\left(x\right)-u\left(y\right)\right).\label{eq:bt7}
\end{equation}

\begin{cor}
If $\min\left(p_{i}\right)>p_{-}$, then $\Delta$ in (\ref{eq:bt7})
has deficiency indices $\left(1,1\right)$, i.e., the non-zero solution
$u$ to $-u=\Delta u$ is in $\mathscr{H}_{E}$; i.e., $0<\left\Vert u\right\Vert _{\mathscr{H}_{E}}<\infty$. \end{cor}
\begin{proof}
The analysis here is analogous to the one above in Example \ref{ex:1d},
and so we omit the details here.\end{proof}
\begin{rem}
For literature on the binomial model and its applications, see for
example \cite{YG06,CLM03,Lon80}.
\end{rem}

\subsection{\label{ex:lad}A 2D lattice}

Let $G=\left(V,E,c\right)$ the graph in Fig \ref{fig:2dl}. Fix $Q,\overline{Q}>1$,
and set the conductance as 
\begin{align*}
a_{n}:=c_{n-1,n} & =Q^{n}\\
\overline{a}_{n}:=\overline{c}_{n-1,n} & =\overline{Q}^{n}
\end{align*}
for all $n\in\mathbb{Z}_{+}$. We get the set of transition probabilities
as follows: 
\[
\begin{cases}
p_{+}=\frac{Q}{1+Q},\;\overline{p}_{+}=\frac{\overline{Q}}{1+\overline{Q}}\\
p_{-}=\frac{1}{1+Q},\;\overline{p}_{-}=\frac{1}{1+\overline{Q}}\\
p_{d}=\mbox{vertical transitions}
\end{cases}
\]

\begin{figure}[H]
\includegraphics[scale=0.5]{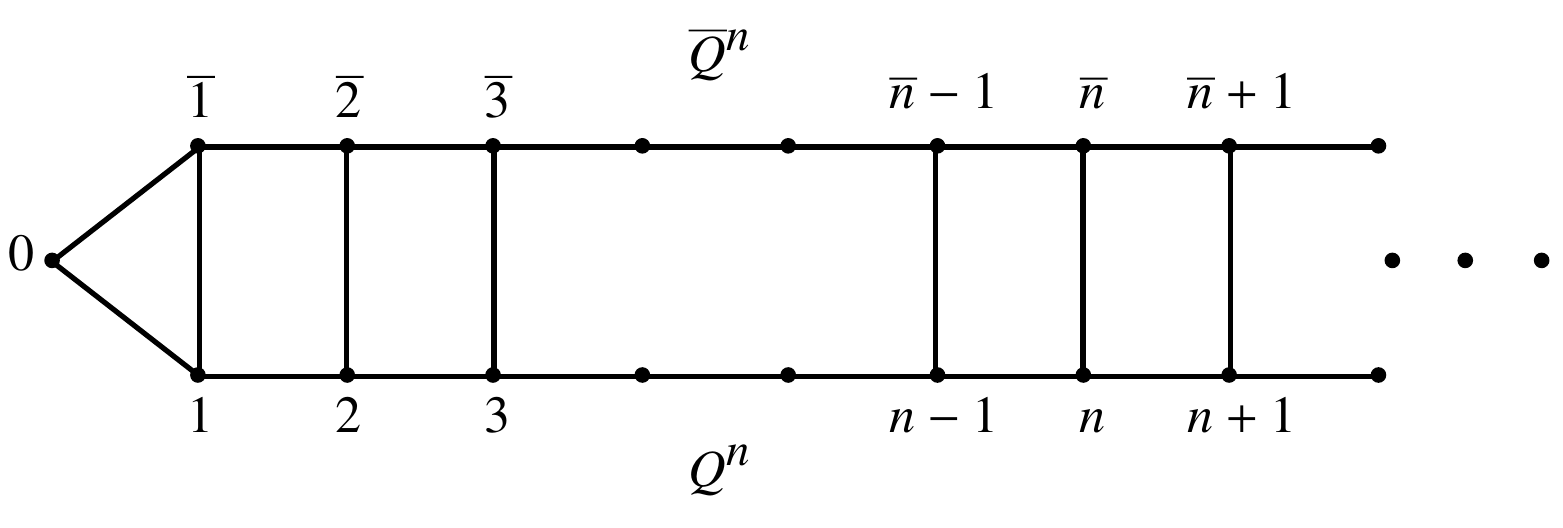}

\protect\caption{\label{fig:2dl}Conductance: $c_{n-1,n}=Q^{n}$ and $\overline{c}_{n-1,n}=\overline{Q}^{n}$}

\end{figure}

\subsection{\label{ex:tri}A Parseval frame that is not an ONB in $\mathscr{H}_{E}$}

Let $c_{01},c_{02},c_{12}$ be positive constants, and assign conductances
on the three edges (see Fig \ref{fig:tri}) in the triangle network.

\begin{figure}[H]
\includegraphics[scale=0.4]{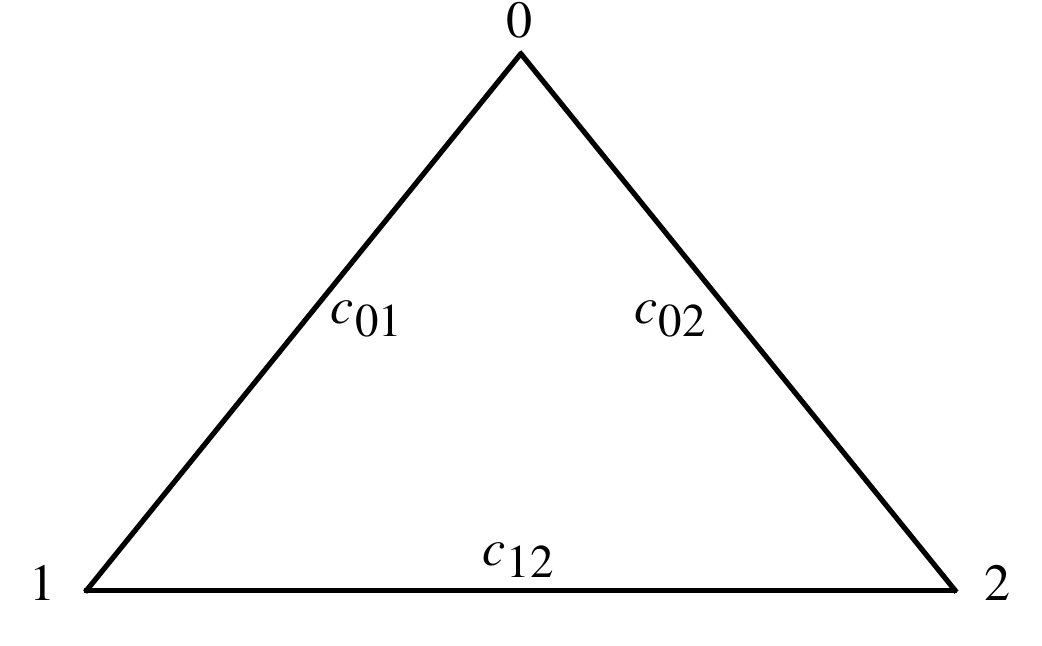}

\protect\caption{\label{fig:tri}The set $\left\{ v_{xy}:\left(xy\right)\in E\right\} $
is not orthogonal.}
\end{figure}

We show that $w_{ij}=\sqrt{e_{ij}}v_{ij}$, $i<j$, in the cyclic
order is a Parseval frame but not an ONB in $\mathscr{H}_{E}$.

Note the corresponding Laplacian $\Delta\left(=\Delta_{c}\right)$
has the following matrix representation 
\begin{equation}
M:=\begin{bmatrix}c\left(0\right) & -c_{01} & -c_{02}\\
-c_{01} & c\left(1\right) & -c_{12}\\
-c_{02} & -c_{12} & c\left(2\right)
\end{bmatrix}\label{eq:LM}
\end{equation}
The dipoles $\left\{ v_{xy}:\left(xy\right)\in E^{\left(ori\right)}\right\} $
as 3-D vectors are the solutions to the equation 
\[
\Delta v_{xy}=\delta_{x}-\delta_{y}.
\]
Hence, 
\begin{align*}
Mv_{01} & =\begin{bmatrix}1 & -1 & 0\end{bmatrix}^{tr}\\
Mv_{02} & =\begin{bmatrix}1 & 0 & -1\end{bmatrix}^{tr}\\
Mv_{12} & =\begin{bmatrix}0 & 1 & -1\end{bmatrix}^{tr}
\end{align*}

We check directly eq. (\ref{eq:en2-2}) holds, and so $\left\{ v_{xy}:\left(xy\right)\in E^{\left(ori\right)}\right\} $
is not orthonormal. For example, we have 
\[
v_{01}=\left[\frac{c_{12}}{c_{01}c_{02}+c_{01}c_{12}+c_{02}c_{12}},-\frac{c_{02}}{c_{01}c_{02}+c_{01}c_{12}+c_{02}c_{12}},0\right]^{tr}
\]
and 
\[
c_{01}\left(v_{01}\left(0\right)-v_{01}\left(1\right)\right)=\frac{c_{01}\left(c_{12}+c_{02}\right)}{c_{01}c_{02}+c_{01}c_{12}+c_{02}c_{12}}<1;
\]
see (\ref{eq:en2-2}). Hence the voltage drop across $\left(01\right)$
is strictly smaller then the $\left(01\right)$ resistance, i.e.,
\[
v_{01}\left(0\right)-v_{01}\left(1\right)<\frac{1}{c_{01}}=\mbox{Res}_{\left(01\right)}.
\]

In this example, the Parseval frame from Lemma \ref{lem:eframe} is
\begin{align*}
w_{01} & =\sqrt{c_{01}}v_{01}=\left[\frac{\sqrt{c_{01}}\, c_{12}}{c_{01}c_{02}+c_{01}c_{12}+c_{02}c_{12}},-\frac{\sqrt{c_{01}}\, c_{02}}{c_{01}c_{02}+c_{01}c_{12}+c_{02}c_{12}},0\right]^{tr}\\
w_{12} & =\sqrt{c_{12}}v_{12}=\left[0,\frac{\sqrt{c_{12}}\, c_{02}}{c_{01}c_{02}+c_{01}c_{12}+c_{02}c_{12}},-\frac{\sqrt{c_{12}}\, c_{01}}{c_{01}c_{02}+c_{01}c_{12}+c_{02}c_{12}}\right]^{tr}\\
w_{20} & =\sqrt{c_{20}}v_{20}=\left[\frac{-\sqrt{c_{20}}\, c_{12}}{c_{01}c_{02}+c_{01}c_{12}+c_{02}c_{12}},0,\frac{\sqrt{c_{20}}\, c_{01}}{c_{01}c_{02}+c_{01}c_{12}+c_{02}c_{12}}\right]^{tr}.
\end{align*}

\begin{rem}
The dipole $v_{xy}$ is unique in $\mathscr{H}_{E}$ as an equivalence
class, not a function on $V$. Note $\ker M$ = harmonic functions
= constant (see (\ref{eq:LM})), and so $v_{xy}+\mbox{const}=v_{xy}$
in $\mathscr{H}_{E}$. Thus, the above frame vectors have non-unique
representations as functions on $V$. Also see Remark \ref{rem:harm}.
\end{rem}
Introduce the vector system of conductance as follows:
\begin{align*}
\widetilde{c}_{1} & =\left(c_{01},c_{01},c_{02}\right)\\
\widetilde{c}_{2} & =\left(c_{02},c_{12},c_{12}\right)\\
\widetilde{c}_{0} & =\left(c_{01},c_{02},c_{12}\right)
\end{align*}
we arrive at the following formula for the spectrum of the system
$\left(V,E,c,\Delta\right)$ where $\left(V,E\right)$ is the triangle
in Fig \ref{fig:tri}. 
\begin{align*}
\lambda_{1} & =0\\
\lambda_{2} & =\mbox{tr}\widetilde{c}_{0}-\sqrt{\left\Vert \widetilde{c}_{0}\right\Vert ^{2}-\left\langle \widetilde{c}_{1},\widetilde{c}_{2}\right\rangle }\\
\lambda_{3} & =\mbox{tr}\widetilde{c}_{0}+\sqrt{\left\Vert \widetilde{c}_{0}\right\Vert ^{2}-\left\langle \widetilde{c}_{1},\widetilde{c}_{2}\right\rangle },
\end{align*}
and so the spectral gap
\[
\lambda_{3}-\lambda_{2}=2\sqrt{\left\Vert \widetilde{c}_{0}\right\Vert ^{2}-\left\langle \widetilde{c}_{1},\widetilde{c}_{2}\right\rangle }
\]
is a function of the coherence for $\widetilde{c}_{1}$ and $\widetilde{c}_{2}$.

\section{Open problems}

Let $G=\left(V,E,c\right)$ be a graph, with vertices $V$, edges
$E$, and conductance $c$; $V$ is countable infinite. The graph-Laplacian
$\Delta$ is essentially selfadjoint as an $l^{2}\left(V\right)$-operator,
but not as a $\mathscr{H}_{E}$-operator. It is known that $\Delta$,
as a $\mathscr{H}_{E}$-operator, has deficiency indices $\left(m,m\right)$,
$m>0$, when the conductance function $c$ is of exponential growth. 
\begin{enumerate}[label=\arabic{enumi}.]
\item Compare the deficiency indices of $\Delta$ (in $\mathscr{H}_{E}$)
for various cases: $V=\mathbb{Z}^{d}$, $d=1,2,3,\ldots$; nearest
neighbor. If $d=1$ must the indices then be $\left(m,m\right)$ with
$m=0$ or 1? Are there examples with $m=2$ in any of the classes
of examples? 
\item What is the spectral representation of the Friedrichs extension $\Delta_{Fri}$
as a $\mathscr{H}_{E}$-operator?  Find the $\mathscr{H}_{E}$ projection
valued measure. 
\item Find the spectrum of the transition operator $P$, considered as an
operator in $l^{2}(V,\widetilde{c})$. When is there point-spectrum?
If so, what are the two top eigenvalues? What is the connection between
this spectrum (the spectrum of $P$), and the spectrum of $\Delta_{Fri}$
in $\mathscr{H}_{E}$, and of $\Delta$ in $l^{2}(V)$?
\item It is not known whether or not the transition operator $P$ is bounded
as an operator of $\mathscr{H}_{E}$ into itself.\end{enumerate}
\begin{acknowledgement*}
The co-authors thank the following for enlightening discussions: Professors
Sergii Bezuglyi, Dorin Dutkay, Paul Muhly, Myung-Sin Song, Wayne Polyzou,
Gestur Olafsson, Robert Niedzialomski, Keri Kornelson, and members
in the Math Physics seminar at the University of Iowa.
\end{acknowledgement*}
\bibliographystyle{amsalpha}
\bibliography{number7}

\end{document}